\documentclass[a4paper,twoside]{article}  

\usepackage{amssymb,amsmath,amsthm,dsfont,amsfonts,color,latexsym}

\usepackage{hyperref}

\usepackage{mathrsfs} 

\definecolor{blue}{rgb}{0.00,0.00,1.00}
\definecolor{red}{rgb}{1.00,0.00,0.00}

\renewcommand{\baselinestretch}{1.2}
\def\bq{\begin{equation}}
\def\eq{\end{equation}}
\def\ba{\begin{array}{ccc}}
\def\bal{\begin{array}{lll}}
\def\ea{\end{array}}

\def\dint{\displaystyle\int} 
 
\def\({\left(}\def\){\right)}
\def\[{\left[}\def\]{\right]}

    \def\intr {\int_{\R^3}}
    
    \def \R    {\mathbb{R}}
    \def \S    {\mathbb{S}}
    \def \N    {\mathbb{N}}
    \def \C    {\mathbb{C}}
    \def \F    {\mathcal{F}}
    \def \O    {\mathbb{O}}

    \def \dt    {\partial_t}
    \def \dx    {\partial_x}
    \def \dxa     {\partial_x^\alpha}
    \def \da      {\partial^\alpha}
    \def \lxtwo   {L_x^2}
    
    \def \lkvtwo  {L_{k,v}^2}
    \def \lxvtwo  {L_{x,v}^2}
    \def \Tdx    {\nabla_x}

       \def\bq{\begin{equation}}
       \def\eq{\end{equation}}
       \def\be{\begin{equation}}
       \def\ee{\end{equation}}
       \def\bma#1\ema{{\allowdisplaybreaks\begin{align}#1\end{align}}}
       \def\bmas#1\emas{{\allowdisplaybreaks\begin{align*}#1\end{align*}}}
       \def\bln#1\eln{{\allowdisplaybreaks\begin{aligned}#1\end{aligned}}}
       \def\nnm{\notag}
       \def\bgr#1\egr{\allowdisplaybreaks\begin{gather}#1\end{gather}}
       \def\bgrs#1\egrs{\allowdisplaybreaks\begin{gather*}#1\end{gather*}}

       \theoremstyle{plain}
       \newtheorem{lem}{\bf Lemma}[section]
       \newtheorem{thm}[lem]{\textbf{Theorem}}

       \newtheorem{remark}[lem]{\bf Remark}

\begin{document}


\title{Spectrum analysis for the relativistic Boltzmann equation}
\author{ Shijia Zhao$^1$,\, Mingying Zhong$^2$\\[2mm]
 \emph
    {\small\it $^1$College of  Mathematics and Information Sciences,
    Guangxi University, China.}\\
    {\small\it E-mail: shijia\_zhao@163.com} \\
    {\small\it  $^2$College of  Mathematics and Information Sciences,
    Guangxi University, China.}\\
    {\small\it E-mail:\ zhongmingying@sina.com}\\[5mm]
    }
\date{ }

 \maketitle

 \thispagestyle{empty}

\begin{abstract}\noindent{
The spectrum structure of the linearized relativistic Boltzmann equation around a global Maxwellian is studied in this paper. Based on the spectrum analysis, we establish the optimal time-convergence rates of the global solution to the Cauchy problem for the relativistic Boltzmann equation.
}

\medskip
 {\bf Key words}. relativistic Boltzmann equation, spectrum analysis, optimal time decay rates.

\medskip
 {\bf 2010 Mathematics Subject Classification}. 76P05, 82C40, 82D05.
\end{abstract}

%
\tableofcontents

\section{Introduction}
\label{sect1}
\setcounter{equation}{0}
In this paper, we consider the Cauchy problem for the relativistic Boltzmann equation
\be\left\{\bln\label{rb1}
 &\partial_t F+\hat{v}\cdot\nabla_x F=Q(F,F),\\
 & F(0,x,v)=F_0(x,v),
\eln\right.\ee
where $F=F(t,x,v)$ is the distribution function with $(t,x,v)\in \R^+\times \R_x^3\times \R_v^3$ and the relativistic velocity $\hat{v}$ is defined by
$$
\hat{v}=\frac{v}{v_0},\quad v_0=\sqrt{1+|v|^2}.
$$
The collision operator $Q(F,G)$ is given by
\bq
Q(F,G)(v)=\int_{\R^3}\int_{\S^2}v_M\sigma(g,\theta)[F(v')G(u')-F(u)G(u)]dud\omega,
\eq
where
\bmas
&\ v_M=\frac{2g\sqrt{1+g^2}}{u_0v_0},\quad u_0=\sqrt{1+|u|^2},\\
&\ 4g^2=2(u_0v_0-uv-1)=s-4,\\
&\ s=2(u_0v_0-uv-1),
\emas
and $d\omega$ is surface measure on the unit sphere $\S^2$. The scattering angle $\theta$ is defined by
$$
\cos\theta=\frac{(U-V)(U'-V')}{(U-V)(U-V)},
$$
with
\bgrs
U=(u_0,u_1,u_2,u_3),\quad V=(v_0,v_1,v_2,v_3),\\
U\cdot V=u_0v_0-\sum_{k=1}^3u_kv_k.
\egrs
Notice that the conservation law of momentum and energy are given by
\bgr
u+v=u'+v', \\
\sqrt{1+|u|^2}+\sqrt{1+|v|^2}=\sqrt{1+|u'|^2}+\sqrt{1+|v'|^2}.
\egr

The method of spectral analysis for the Relativistic Boltzamann equation is similar to that of Boltzmann equation \cite{Ellis,Ukai2,Ukai1}. Without loss of generality, a global relativistic Maxwellian takes the form
$$
M(v)=e^{-\sqrt{1+|v|^2}}.
$$
Set the perturbation $f(t,x,v)$ of $F(t,x,v)$ around $M(v)$ by
$$F=M+\sqrt{M}f.$$
Then the RB \eqref{rb1} becomes
\be \left\{\begin{aligned}\label{rb2}
 &\ \dt f+\hat{v}\cdot\nabla_xf-Lf= \Gamma(f,f),\\
 &\ f(0,x,v)=f_0(x,v),
\end{aligned}\right.
\ee
where the linearized collision operator $Lf$ and nonlinear term $\Gamma(f,f)$ are defined by
\bma
Lf&=\frac1{\sqrt M}[Q(M,\sqrt{M}f)+Q(\sqrt{M}f,M)],  \label{Lf}\\
\Gamma(f,f)&=\frac1{\sqrt M}Q(\sqrt{M}f,\sqrt{M}f).   \label{gf}
 \ema
The linearized collision operator $L$ can be written as \cite{Glassey3,Yang1}
$$
Lf=-\nu(v)f+Kf,
$$
where $\nu(v)$ is collision frequency and its expression is
$$
\nu(v)=\int_{\R^3}\int_{\S^2}v_M\sigma(g,\theta)\mu(u)d\omega du,
$$
and $K$ is a compact operator on $L^2(\R^3_v)$ and its expression is
$$
Kf=\int_{\R^3}\int_{\S^2}v_M\sigma(g,\theta)M^{\frac{1}{2}}[M^{\frac{1}{2}}(u')f(v')+M^{\frac{1}{2}}(v')f(u')-M^{\frac{1}{2}}(v)f(u)]d\omega du.
$$
The scattering kernel $\sigma$ satisfies the following condition \cite{Glassey3}
\bq\label{ker}
C_1\frac{g^{\beta+1}}{1+g}\sin^{\gamma}\theta\leq\sigma(g,\theta)\leq C_2(g^\beta+g^\delta)\sin^{\gamma}\theta,
\eq
where $C_1$ and $C_2$ are positive constants,
$0\leq\delta<\frac{1}{2}$, $0\leq\beta<2-2\delta$, and either
$\gamma\geq0$ or
$$|\gamma|\leq \min\left\{2-\beta,\frac{1}{2}-\theta,\frac{1}{3}(2-2\theta-\beta)\right\},$$

In fact, $L$ is a non-positive and self-adjoint  operator and its null space, denoted by $N_0$, is 5 dimensional subspace given by
\bq
N_0={\rm span}\{\sqrt{M},\, v_j \sqrt{M},\,\, 1\leq j \leq 3,\, v_0 \sqrt{M}\}.
\eq
We normalized the elements of $N_0$ as
\bq\label{orthonormal basis}
\psi_0 = p_0^{-\frac12}\sqrt{M},\quad \psi_1 = p_1^{-\frac12}v_j\sqrt{M},\quad \psi_4 = p_3^{-\frac12}(v_0 - p_2)\sqrt{M}.
\eq
Here the normalized constants are given by
$$
p_0=\intr M(v)dv, \quad p_1=\intr v^2_jM(v)dv, \quad p_2=\intr v_0M(v)dv, \quad p_3= \intr v_0^2Mdv-p_2^2 .
$$

Set $L^2(\R^3)$ be a Hilbert space of complex-value functions $f(v)$ on $\R^3$ with the inner product and the norm
$$
(f,g)=\int_{\R^3}f(v)\overline{g(v)}dv,\quad \|f\|=\(\int_{\R_3}|f(v)|^2dv\)^{1/2}.
$$

Define $P_0$ be the projection operator from $L^2(\R_v^3)$ to the null $N_0$ and $P_1=I-P_0$.
Corresponding to the linearized operator $L$, it is shown in \cite{Glassey3}  that
there exists a constant $\mu> 0$ such that
\be (Lf,f)\le -\mu( P_1f,P_1f), \quad \forall f\in D(L),\ee
where $D(L)$ is the domain of $L$ given by
$$
 D(L)=\left\{f\in L^2(\R^3)\,|\,\nu(v)f\in L^2(\R^3)\right\}.
$$

Note that the solution $f$ can be decomposed into the macroscopical part and microscopical part as
\be\label{macro micro}\left\{\bal
f=P_0f+P_1f,\\
P_0f=n\psi_0+\sum\limits_{j=1}^3m_j\psi_j+q\psi_4,
\ea\right.\ee
where the density $n$, the momentum $m_j=(m_1, m_2, m_3)$ and the energy $q$ are defined by
\bq
(f,\psi_0)=n,\quad (f,\psi_j)=m_j,\quad (f,\psi_4)=q.
\eq

To obtain the optimal decay rate of global solution to \eqref{rb2}, we need to consider Cauchy problem for the linear relativistic Boltzmann equation
\be\label{lrb}\left\{\bal
\partial_tf=Bf,\quad t>0,\\
f(0,x,v)=f_0(x,v),\quad (x,v)\in \R_x^3\times \R_v^3,
\ea\right.\ee
where the operator $B$ is defined by
\be Bf=Lf-\hat{v}\cdot\Tdx f. \ee

We take the Fourier transform in \eqref{lrb} with respect to $x$ to get
\be\label{lrb1}\left\{\bal
\partial_t\hat{f} =\hat{B}(k)\hat{f},\quad t>0,\\
\hat{f}(0,k,v)=\hat{f}_0(k,v),  \quad (k,v)\in\R^3_{k}\times\R_v^3,
\ea\right.\ee
where  the operator $\hat{B}(k)$ is defined by
\be\label{B}
 \hat{B}(k)f=(L-ik\cdot\hat{v})f .
\ee

\noindent\textbf{Notations:} \ \
Before state the main results in this paper, we list some notations.
Define the Fourier transform of $f=f(x,v)$ by
$$\hat{f}(k,v)=\mathcal{F}f(k,v)=\frac1{(2\pi)^{3/2}}\intr f(x,v)e^{- i x\cdot k}dx,$$
where and throughout this paper we denote $i=\sqrt{-1}$.

Let us introduce a Sobolev space of function $f=f(x,v)$ by $H^N=L^2(\R_v^3;H^N(\R_x^3))$ with the norm
$$\|f\|_{H^N} =\(\intr\intr(1+|k|^2)^N|\hat{f}(k,v)|^2dkdv\)^{1/2}.$$
For $l\in\R$, we define a weight function by
$$w_l(v)=(\sqrt{1+|v|^2})^l,$$
and the Sobolev spaces $H_{N,l}$ as
$$H_{N,l}=\{f\in L^2(\R^3_{x}\times \R^3_{v})\mid\|f\|_{H_{N,l}}<\infty\},$$
with the norm
$$\|f\|_{H_{N,l}}=\sum\limits_{|\alpha|\le N}\|w_l(v)\da_x f\|_{L^2(\R^3_{x}\times \R^3_{v})}.$$
We also define a space of function $f=f(x,v)$ by $Z_q=L^2(\R_v^3;L^q(\R_x^3))$ with the norm
$$\|f\|_{Z_q}=\(\intr\(\intr|f(x,v)|^qdx\)^{2/q}dv\)^{1/2}.$$
Throughout this paper, denote
$\dx^\alpha=\partial_{x_1}^{\alpha_1}\partial_{x_2}^{\alpha_2}\partial_{x_3}^{\alpha_3}$ and $k^\alpha=k_1^{\alpha_1}k_2^{\alpha_2}k_3^{\alpha_3}$ with
$k=(k_1, k_2, k_3)$ and $\alpha=(\alpha_1,\alpha_2, \alpha_3)$. We denote by $\|\cdot\|_{L_{x,v}^2}$ and $\|\cdot\|_{\lkvtwo}$ the norms of the function spaces $L^2(\R_x^3\times\R_v^3)$ and $L^2(\R_k^3\times\R_v^3)$, respectively, and denote by $\|\cdot\|_{L^2_x}$, $\|\cdot\|_{L^2_v}$ and $\|\cdot\|_{L^2_k}$ the norms of the function spaces
$L^2(\R_x^3),\, L^2(\R_v^3)$ and $L^2(\R_k^3)$, respectively.

The main results can be states as follows. First, we have the spectrum set of the linear operator $\hat{B}(k)$ as follows.

\begin{thm}\label{global spectrum}
(1) For any $\tau_1>0$, there exists $d=d(\tau_1)>0$ so that for $|k|>\tau_1$,
$$
\sigma(\hat{B}(k))\subset\{\lambda\mid \mathrm{Re}\lambda<-d \}.
$$
(2) There exists a small constant $\tau_0>0$  so that for  $|k|\le\tau_0$,
$$
\sigma(\hat{B}(k))\cap\{\lambda\in\C \mid \mathrm{Re}\lambda>-\mu/2\}=\{\lambda_j(|k|)\}_{j=-1}^3,
$$
where $\lambda_j(|k|)$ is $C^\infty$ function of $|k|$. In particular, the eigenvalues $ \lambda_j(|k|) $ admit the following asymptotic expansion for $|k|\le\tau_0$,
\be\left\{\bln
&\lambda_{\pm1}(|k|)=\pm i\sqrt{a^2+b^2}|k|+A_{\pm1}|k|^2+O(|k|^3),\\
&\lambda_0(|k|)=A_{0}|k|^2+O(|k|^3),\\
&\lambda_2(|k|)=\lambda_3(|k|)=A_{2}|k|^2+O(|k|^3),
\eln\right.\ee
where $a,b>0$ and $A_j>0$, $j=-1,0,1,2$ are  constants given by \eqref{a}.
\end{thm}

The, we have the time-asymptotical decay rates of global solution to the nonlinear Boltzmann equation \eqref{rb2} as follows.
\begin{thm}\label{nonlinear upper}
Assume that $f_0\in H_{N,1}\cap Z_1$ for $N\ge4$,  and $\|f_0\|_{ H_{N,1}\cap Z_1}\le\delta_0$ for a constant $\delta_0>0$ small enough. Then, there exists a globally unique solution $f=f(t,x,v)$ to the relativistic Boltzmann equation \eqref{rb2}, which satisfies
\bma
&\ \|\dxa(f(t),\psi_j)\|_{\lxtwo}\le C\delta(1+t)^{-\frac{3}{4}-\frac{|\alpha|}{2}},\quad j=0,1,2,3,4,\label{nonlinear upper1}\\
&\ \|\dxa P_1f(t)\|_{\lxvtwo}\le C\delta(1+t)^{-\frac{5}{4}-\frac{|\alpha|}{2}},\label{nonlinear upper2}\\
&\ \|P_1f\|_{ H_{N,1}}+\|\nabla_xP_0f\|_{H^{N-1}}\le C\delta(1+t)^{-\frac{5}{4}},\label{nonlinear upper3}
\ema
for $|\alpha|=0,1$. If  it further holds that $P_0f_0=0$, then
\bma
&\ \|\dxa(f(t),\psi_j)\|_{\lxtwo}\le C\delta(1+t)^{-\frac{5}{4}-\frac{|\alpha|}{2}},\quad j=0,1,2,3,4,\label{nonlinear upper4}\\
&\ \|\dxa P_1f(t)\|_{\lxvtwo}\le C\delta(1+t)^{-\frac{7}{4}-\frac{|\alpha|}{2}},\label{nonlinear upper5}\\
&\ \|P_1f\|_{ H_{N,1}}+\|\nabla_xP_0f\|_{H^{N-1}}\le C\delta(1+t)^{-\frac{7}{4}},\label{nonlinear upper6}
\ema  for $|\alpha|=0,1$.
\end{thm}

We can also prove that the above time-decay rates are indeed optimal in following sense.
\begin{thm}\label{nonlinear lower}
Let the assumptions of Theorem \ref{nonlinear upper} hold. Assume further that there exist two constants $d_0,d_1>0$ and a small constant $\tau_0>0$ such that $\inf_{|k|\le \tau_0}|(\hat{f}_0,\psi_0)|\ge d_0$, $ \sup_{|k|\le \tau_0}|(\hat{f}_0,b\psi_4-a\psi_0)|=0$  and $\sup_{|k|\le \tau_0}|(\hat{f}_0,\psi')|=0$  with $\psi'=(\psi_1,\psi_2.\psi_3)$. Then, the global solution $f$ to the relativistic Boltzmann equation \eqref{rb2} satisfies
\bma
&\ C_1\delta_0(1+t)^{-\frac{3}{4}}\le\|(f(t),\psi_j)\|_{\lxtwo}\le C_2\delta_0(1+t)^{-\frac{3}{4}},\quad j=0, 4,\label{nonlinear lower1}\\
&\ C_1\delta_0(1+t)^{-\frac{3}{4}}\le\|(f(t),\psi')\|_{\lxtwo}\le C_2\delta_0(1+t)^{-\frac{3}{4}}, \label{nonlinear lower1a}\\
&\ C_1\delta_0(1+t)^{-\frac{5}{4}}\le\|P_1f(t)\|_{\lxvtwo}\le C_2\delta_0(1+t)^{-\frac{5}{4}},\label{nonlinear lower2} \\
&\ C_1\delta_0(1+t)^{-\frac{3}{4}}\le\|f(t)\|_{H_{N,1}}\le C_2\delta_0(1+t)^{-\frac{3}{4}},\label{nonlinear lower3}
\ema
for $t>0$ large with constants $C_2>C_1>0$.
If it holds that $P_0f_0=0$, $\inf_{|k|\le \tau_0}|(\hat{f}_0,L^{-1}P_1(\hat{v}\cdot\omega)(\omega\cdot\psi'))|\ge d_0$ and $\sup_{|k|\le \tau_0}|(\hat{f}_0,L^{-1}P_1(\hat{v}\cdot\omega)\psi_j)|=0$ with $j=0,4$ and $\omega=k/|k|$, then
\bma
&\ C_1\delta_0(1+t)^{-\frac{5}{4}}\le\|(f(t),\psi_j)\|_{\lxtwo}\le C_2\delta_0(1+t)^{-\frac{5}{4}},\quad j=0, 4,\label{nonlinear lower4}\\
&\ C_1\delta_0(1+t)^{-\frac{5}{4}}\le\|(f(t),\psi')\|_{\lxtwo}\le C_2\delta_0(1+t)^{-\frac{5}{4}}, \label{nonlinear lower4a}\\
&\ C_1\delta_0(1+t)^{-\frac{7}{4}}\le\|P_1f(t)\|_{\lxvtwo}\le C_2\delta_0(1+t)^{-\frac{7}{4}},\label{nonlinear lower5} \\
&\ C_1\delta_0(1+t)^{-\frac{5}{4}}\le\|f(t)\|_{H_{N,1}}\le C_2\delta_0(1+t)^{-\frac{5}{4}},\label{nonlinear lower6}
\ema
for $t>0$ large.
\end{thm}
There have been many important achievements on the existence and long time behaviors of solutions to the classical Boltzmann equation. In particular, in \cite{DiPerna} it has been shown that the global existence of renormalized weak solution subject to general large initial data. The global existence and optimal decay rate $(1+t)^{-\frac{3}{4}}$ of strong solution near Maxwellian for hard potential was proved in \cite{Guo1,Liu2,Ukai2,Ukai1,Ukai3,Zhong}. The global existence of solutions near vacuum was investigated in \cite{Bellomo,Glassey4,Illner}.
The spectrum structure of classical Boltzmann equation has been developed vigorously. The spectrum of the Boltzmann equation with hard sphere and hard potentials was constructed in \cite{Ellis,Ukai2,Ukai1,Ukai3}, and the optimal time decay rate based on spectral analysis  has been established in the torus \cite{Ukai2}, and in $\R^n$ \cite{Ukai1,Ukai3}. The pointwise behavior of the Green function of the Boltzmann equation was verified in \cite{Liu1,Liu3} for hard sphere, and in \cite{Lee,Lin} for hard potentials and soft potentials.

There have been a lot of works on the relativistic Boltzmann equation, see \cite{Cer1,Dudy,Dudy2,Glassey1,Glassey2,Glassey3,Hsiao} and the references therein. The background of relativistic Boltzmann equations is mentioned in \cite{Cer1}. The existence and uniqueness of the solution to the linearized relativistic Boltzmann equation  has been proved  in \cite{Dudy}. The global solution of relativistic Boltzmann equation near a relativistic Maxwellian  is obtained in \cite{Glassey2,Glassey3}  for hard potentials, and in \cite{Duan1,Strain}  for soft potentials. Yang and Yu established in \cite{Yang1} that the global solutions to the relativistic Boltzmann and Landau equations tend to the equilibriums at $(1+t)^{-\frac{3}{4}}$ in $L^2$-norm by using compensating function and the energy method. 

The organization of this paper is as follows. In section 2, we study  the spectrum and resolvent sets of the linearized relativistic Boltzmann equation and show asymptotic expansions of eigenvalues and eigenfunctions of the linearized operator $\hat{B}(k)$ at low frequency based on the approach in \cite{Ukai1}. In section 3, we study the semigroup $e^{t\hat{B}(k)}$ generated by the linear operator $\hat{B}(k)$, and we establish the optimal time decay rates of the global solution to the linearized relativistic Boltzmann equation in terms of the $e^{t\hat{B}(k)}$ by drawing on the idea of \cite{Zhong}. In section 4, based on the asymptotic behaviors of linearized problem, we establish the optimal time decay rates of the original nonlinear relativistic Boltzmann equation.


\section{Spectral analysis}
\label{sect2}
\setcounter{equation}{0}
In this section, we study the spectrum and resolvent sets of linear collision operator $\hat{B}(k)$ defined by \eqref{B}, which will be applied to study the optimal decay rate of solution to the linear system \eqref{lrb}.

We review some properties of the operators $\nu(v)$ and $K$.
\begin{lem}[\cite{Dudy,Glassey4}]\label{Kbounded}
Under the assumption \eqref{ker}, the following holds.
\begin{enumerate}
\item[{\rm (i)}]
There are two constant $c_0,c_1>0$ such that
\bq\label{nu}
c_0v_0^{\beta/2}\le\nu(v)\le c_1 v_0^{\beta/2},\quad v\in \R^3.
\eq

\item[{\rm (ii)}] The operator K is bounded and compact in $L^2(\R_v^3)$. Furthermore, K is an integral operator
\bq
Kf(v)=\intr K(u,v)f(u)du
\eq
with the kernel $K(u,v)$ satisfying

\begin{enumerate}
\item[{\rm (1)}] $\sup\limits_v\dint_{\R^3} |K(u,v)|du<\infty$,

\item[{\rm (2)}] $\sup\limits_v\dint_{\R^3}  |K(u,v)|^2du<\infty$,

\item[{\rm (3)}] $\dint_{\R^3}|K(u,v)|(1+|u|^2)^{-\alpha/2}du\le C(1+|v|^2)^{-\frac{1}{2}(\alpha+\eta)}$ for any $\alpha\ge0$, where
$$
\eta=1-\frac{1}{2}[3|r|+\beta+2\delta]>0.
$$
\end{enumerate}
\end{enumerate}
\end{lem}


Let $\hat{A}(k)$ be the operator obtained by dropping $K$ from $\hat{B}(k)$, we have
\bma\label{A}
& \hat{A}(k)f=(-ik\cdot\hat{v}-\nu(v))f, \quad f\in D(\hat{A}(k)),\\
& D(\hat{A}(k))=D(\hat{B}(k))=\left\{f\in L^2(\R^3)\,|\,\nu(v)f\in L^2(\R^3)\right\}.\nnm
\ema

\begin{lem} \label{spectrumA}
The operator $\hat{A}(k)$ generates a continuous contraction semigroup on $L^2(\R^3_v)$, which satisfies
$$
 \|e^{t\hat{A}(k)}f\|\le e^{-c_0t}\|f\|,\quad t>0,\,\, f\in L^2(\R_v^3).
$$
In addition, we have $\sigma(\hat{A}(k))\subset\{\lambda\mid {\rm Re}\lambda\le-c_0\}$.
\end{lem}
\begin{proof}By  \eqref{nu} and \eqref{A}, we have for any $f\in D(\hat{A}(k))$ that
\be
   {\rm Re}(\hat{A}(k)f,f)= {\rm Re}(\hat{A}(-k)f,f)=-(\nu f,f)
   \le -c_0\|f\|^2,
\ee
which proves that $\hat{A}(k)$ and $\hat{A}(-k)$ are dissipative operators on $L^2(\R^3_v)$.

 Note that $\hat{A}(k)$ is densely defined in $L^2(\R^3_v)$,   and the adjoint operator $\hat{A}(-k)$ of $\hat{A}(k)$ is also densely defined in $L^2(\R^3_v)$. This implies that $\hat{A}(k)$ is a closed operator in $L^2(\R^3_v)$ \cite{Reed}. Thus, it
follows from Corollary 4.4 on p.15 of \cite{Pazy} that $\hat{A}(k)$ generates a continuous contraction semigroup $e^{t\hat{A}(k)}$ on $L^2(\R^3_v)$. Then, by direct computation, we can show for any $t\ge0$ that
 \bq
 \|e^{t\hat{A}(k)}f\|\le e^{-c_0t}\|f\|.
 \eq
 By the semigroup theory, we find
 $$\{\lambda:{\rm Re}\lambda>-c_0\}\subset\rho(\hat{A}(k)).$$
 Thus we have
 $$\sigma(\hat{A}(k))\subset\{\lambda:{\rm Re}\lambda\le-c_0\}.$$
 This proves the lemma.
\end{proof}

\begin{lem}\label{B strongly contraction semigroup}
The operator $\hat{B}(k)$ generates a continuous contraction semigroup on $L^2(\R^3_v)$, which satisfies
\bq
\|e^{t\hat{B}(k)}f\|\le\|f\|,\quad   t>0, \,\,\, f\in L^2(\R_v^3).
\eq
\end{lem}
\begin{proof} We can compute that $\hat{B}(k)$ and $\hat{B}(-k)$ are dissipative. By \eqref{B}, we have
\bq
(\hat{B}(k)f,g)=((L-ik\cdot\hat{v})f,g)=(f,(L+ik\cdot\hat{v})g)=(f,\hat{B}(-k)g).
\eq
Thus
 \bq
 \mathrm{Re}(\hat{B}(k)f,f)= \mathrm{Re}(\hat{B}(-k)f,f)=(Lf,f)\le0,\quad \forall f\in D(\hat{B}(k)).
 \eq
 Note that $\hat{B}(k)$ is a densely defined closed operator in $L^2(\R_v^3)$. Hence, $\hat{B}(k)$ generates a continuous contraction semigroup in $L^2(\R_v^3)$.
\end{proof}

\begin{lem}\label{Specturm set}The following conditions hold for all $k\in \R^3$.
 \begin{enumerate}
\item[(1)]
$\sigma_{ess}(\hat{B}(k))\subset \{\lambda\in \mathbb{C}\,|\, {\rm Re}\lambda\le -c_0\}$ and $\sigma(\hat{B}(k))\cap \{\lambda\in \mathbb{C}\,|\, -c_0<{\rm Re}\lambda\le 0\}\subset \sigma_{d}(\hat{B}(k))$.
\item[(2)]
 If $\lambda(k)$ is an eigenvalue of $\hat{B}(k)$, then ${\rm Re}\lambda(k)<0$ for any $|k|\ne 0$ and $ \lambda(k)=0$ iff $|k|=0$.
 \end{enumerate}
\end{lem}

\begin{proof}
By  \eqref{nu} and \eqref{A},  $\lambda-\hat{A}(k)$ is invertible for ${\rm
Re}\lambda>-c_0$. Since $K$ is a compact operator  on $L^2 (\R^3_v)$, $\hat{B}(k)$ is a compact perturbation of $\hat{A}(k)$, and so, thanks
to Theorem 5.35 in p.244 of \cite{Kato}, $\hat{B}(k)$ and $\hat{A}(k)$
have the same essential spectrum, namely, $\sigma_{ess}(\hat{B}(k))=\sigma_{ess}(\hat{A}(k))\subset \{\lambda\in \mathbb{C}\,|\, {\rm Re}\le -c_0\}$. Thus the spectrum of
$\hat{B}(k)$ in the domain ${\rm Re}\lambda>-c_0$ consists of discrete eigenvalues  with possible accumulation
points only on the line ${\rm Re}\lambda= -c_0$. This proves (1).

Next, we prove (2) as follows. Indeed, let $u\ne 0$ be the eigenfunction  of $\hat{B}(k)$ corresponding to the eigenvalue $\lambda$ so that
\bq
\hat{B}(k)u=\lambda u,
\eq
that is
\bq\label{B eigenvalue 1}
(L-ik\cdot\hat{v})u=\lambda u.
\eq
Taking the inner product with $u$ and choosing the real part, we have
\bq\label{B eigenvalue 2}
{\rm Re}\lambda\|u\|^2=(Lu,u) .
\eq
We know that ${\rm Re}\lambda\le0$ from $L$ is non-positive operator. Suppose that there is an eigenvalue $\lambda$ such that ${\rm Re}\lambda=0$ holds, then we have $(Lu,u)=0$ from \eqref{B eigenvalue 2}, which implies that $u\in N_0$. Thus, the eigenvalue problem \eqref{B eigenvalue 1} is transformed to
\bq\label{B eigenvalue 3}
(-ik\cdot\hat{v})u=({\rm Im}\lambda)u.
\eq
It follow that \eqref{B eigenvalue 3} holds if and only if $k=0$ and ${\rm Im}\lambda=0$.
\end{proof}

For ${\rm Re}\lambda>-c_0$, we decompose the
operator $\hat{B}(k)$ into
\be
\lambda-\hat{B}(k)=\lambda-\hat{A}(k)-K=(I-K(\lambda-\hat{A}(k))^{-1})(\lambda-\hat{A}(k)), \label{B_d}
\ee
and estimate the right hand terms of \eqref{B_d} as follows.

\begin{lem}\label{KAbounded}
For any $\delta>0$, we have
 \bq\label{KAbounded1}
 \sup_{{\rm Re}\lambda\ge-c_0+\delta,{\rm Im}\lambda\in\R}\|K(\lambda-\hat{A}(k))^{-1}\|\le C\delta^{-1+\frac{\eta}{3+2\eta}}(1+|k|)^{-\frac{\eta}{3+2\eta}}.
 \eq
For any $\delta$, $\tau_0>0$, there is a constant $\tau_0=2\tau_0>0$ such that if $|{\rm Im}\lambda|\ge \tau_0$, then
  \bq\label{KAbounded2}
 \sup_{{\rm Re}\lambda\ge-c_0+\delta,|k|\le\tau_0}\|K(\lambda-\hat{A}(k))^{-1}\|\le C\delta^{-1-\frac{3}{2\eta+3}}(1+|{\rm Im}\lambda|)^{-\frac{2\eta}{2\eta+3}}.
 \eq
\end{lem}
\begin{proof}
Choosing $R>1$, we decompose
\bma
\|K(\lambda-\hat{A}(k))^{-1}f\|^2& \le 2\int_{\R^3}\bigg(\int_{|u|\le R}K(v,u)(\lambda+\nu(u)+ik\cdot\hat{u})^{-1}f(u)du\bigg)^2dv\nnm\\
&+2\int_{\R^3}\bigg(\int_{|u|\ge R}K(v,u)(\lambda+\nu(u)+ik\cdot\hat{u})^{-1}f(u)du\bigg)^2dv\nnm\\
&=I_1+I_2.
\ema

For $I_1$, we have by Lemma \ref{Kbounded} (3)
\bma
I_1& \le2\int_{\R^3}\int_{|u|\le R}K^2(v,u)f^2(u)du\int_{|u|\le R}(\lambda+\nu(u)+ik\cdot\hat{u})^{-2}dudv\nnm\\
& \le C\int_{|u|\le R}(\lambda+\nu(u)+ik\cdot\hat{u})^{-2}du\|f\|^2.
\ema
We choose $\mathbb{O}$ to be a rotation transform of $\R^3$ satisfying $\mathbb{O}^Tk\rightarrow(0,0,|k|)$, we have
\bma
\int_{|u|\le R}|\lambda+\nu(u)+ik\cdot\hat{u}|^{-2}du& \le\int_{|u|\le R}\frac{1}{({\rm Re}\lambda+c_0)^2+({\rm Im}\lambda+\hat{u}\cdot k)^2}du\nnm\\
&=\int_{|u|\le R}\frac{1}{({\rm Re}\lambda+c_0)^2+({\rm Im}\lambda+\hat{u}_3 |k|)^2}du\nnm\\
&=\int_0^Rdr\int_0^{2\pi}d\theta\int_0^\pi\frac{1}{({\rm Re}\lambda+c_0)^2+({\rm Im}\lambda+|k|\frac{r}{\sqrt{1+r^2}}\cos{\varphi})^2}r^2\sin{\varphi}d\varphi\nnm\\
&=\int_0^Rdr\int_0^{2\pi}d\theta\int_{-1}^1\frac{1}{({\rm Re}\lambda+c_0)^2+({\rm Im}\lambda+|k|\frac{rt}{\sqrt{1+r^2}})^2}r^2dt\nnm\\
&=2\pi\int_0^R\bigg(\int_{{\rm Im}\lambda-\frac{|k|r}{\sqrt{1+r^2}}}^{{\rm Im}\lambda+\frac{|k|r}{\sqrt{1+r^2}}}\frac{1}{({\rm Re}\lambda+c_0)^2+z^2}\cdot\frac{\sqrt{1+r^2}}{r|k|}\cdot r^2dz\bigg)dr\nnm\\
&\le C\int_0^R\frac{\sqrt{1+r^2}}{r|k|}\cdot r^2dr\nnm\\
&\le C|k|^{-1}\int_0^R1+r^2dr\nnm\\
&\le CR^3|k|^{-1},\quad R>1.
\ema

For $I_2$, we have by Lemma \ref{Kbounded} (4)
\be
I_2\le C\delta^{-2}(1+|u|^2)^{-\eta}\le C\delta^{-2}R^{-2\eta}, \quad (\eta>0).
\ee
Choose $R=(|k|/\delta)^{\frac{1}{3+2\eta}}$, we can verify \eqref{KAbounded1}.

If $|k|\le \tau_0$, $|u|\le R$ and $|{\rm Im}\lambda|\ge 2\tau_0$, we have
\bq\label{KAbounded 5}
|{\rm Im}\lambda+k\cdot\hat{v}|\ge|{\rm Im}\lambda|-|k||\hat{v}|\ge|{\rm Im}\lambda|-\frac{\tau_0R}{\sqrt{1+R^2}}\ge \frac{|{\rm Im}\lambda|}{2}.
\eq
Hence
\bma
I_1&\ \le C\int_{|u|\le R}|\lambda+ik\cdot\hat{u}+\nu(u)|^{-2}du\nnm\\
&\ \le C(\delta^2+|{\rm Im}\lambda|^2)^{-1}R^3.
\ema
By taking $R=(|{\rm Im}\lambda|/\delta)^{\frac{1}{\eta+\frac{3}{2}}}$, we can verify \eqref{KAbounded2}.
\end{proof}

\begin{thm}\label{high frequency spectrum}
(1) For any $\tau_0>0$, there exists $d(\tau_0)>0$ such that when $|k|>\tau_0$,
\bq
\sigma(\hat{B}(k))\subset\{\lambda\in \C\mid {\rm Re}\lambda<-d(\tau_0)\}.
\eq

(2) For any $\delta>0$ and all $k\in \R^3$, there exists $y_1=y_1(\delta)>0$ such that
\bq
\rho(\hat{B}(k))\supset\{\lambda\in\C \mid {\rm Re}\lambda\ge-c_0+\delta,\, |{\rm Im}\lambda|\ge y_1\}\cup\{\lambda\in\C \mid {\rm Re}\lambda>0\}. \label{resov}
\eq
\end{thm}

\begin{proof}
We first show that $\sup_{k\in\R^3}|{\rm Im}\lambda|<\infty$ for any $\lambda\in\sigma(\hat{B}(k))\cap \{\lambda\in \C\mid {\rm Re}\lambda\ge-c_0+\delta\}$. By Lemma \ref{KAbounded}, there exists $\tau_1=\tau_1(\delta)>0$ large enough so that for ${\rm Re}\lambda\ge-c_0+\delta$ and $|k|\ge\tau_1$,
\bq\label{KAleq1/2}
\|K(\lambda-\hat{A}(k))^{-1}\|\le\frac{1}{2}.
\eq
This implies that the operator $I-K(\lambda-\hat{A}(k))^{-1}$ is invertible on $L^2(\R^3_v)$. Since operator $\lambda-\hat{B}(k)$  has the following decomposition for ${\rm Re}\lambda>-c_0$,
$$
\lambda-\hat{B}(k)=\lambda-\hat{A}(k)-K=(I-K(\lambda-\hat{A}(k))^{-1})(\lambda-\hat{A}(k)),
$$
it follows that $\lambda-\hat{B}(k)$ is also invertible on $L^2(\R^3_v)$ for ${\rm Re}\lambda\ge-c_0+\delta$ and $|k|\ge\tau_1$, and it satisfies
\bq\label{Reeigenvalue neg0 2}
(\lambda-\hat{B}(k))^{-1}=(\lambda-\hat{A}(k))^{-1}(I-K(\lambda-\hat{A}(k))^{-1})^{-1},\quad |k|\ge\tau_1,
\eq
namely,
\bq\label{Reeigenvalue neg0 4}
\rho(\hat{B}(k))\supset\{\lambda\in\C \mid {\rm Re}\lambda\ge-c_0+\delta\},\quad  |k|\ge\tau_1.
\eq

For $|k|\le\tau_1$, by Lemma \ref{KAbounded}, there exists $y_1=y_1(\tau_1,\delta)>0$ such that \eqref{KAleq1/2} holds for $|{\rm Im}\lambda|>y_1$. This also implies the invertibility of $\lambda-\hat{B}(k)$ for $|k|\le\tau_1$ and $|{\rm Im}\lambda|>y_1$, namely,
\bq\label{Reeigenvalue neg0 5}
\rho(\hat{B}(k))\supset\{\lambda\in\C\mid {\rm Re}\lambda\ge-c_0+\delta,\, |{\rm Im}\lambda|>y_1\},\quad |k|\le\tau_1.
\eq
By \eqref{Reeigenvalue neg0 4} and \eqref{Reeigenvalue neg0 5}, we obtain \eqref{resov} and hence
\bq
\sigma(\hat{B}(k))\cap\{\lambda\in\C\mid {\rm Re}\lambda\ge-c_0+\delta\}\subset\{\lambda\in\C\mid {\rm Re}\lambda\ge-c_0+\delta,\, |{\rm Im}\lambda|\le y_1\}.
\eq

Next, we want to show that $\sup_{|k|>\tau_0}{\rm Re}\lambda(k)<0$ for any $\lambda\in\sigma(\hat{B}(k))\cap \{\lambda\in \C\mid {\rm Re}\lambda\ge-c_0+\delta\}$. By \eqref{Reeigenvalue neg0 4}, it holds that $\sup_{|k|\ge\tau_1}{\rm Re}\lambda(k)<0$. Hence, it is sufficient to prove that $\sup_{\tau_0< |k|<\tau_1}{\rm Re}\lambda(k)<0$. If it does not hold, there exists a sequence of $\{k_n, \lambda_n, \phi_n\}$ satisfying $|k_n|\in(\tau_0,\tau_1)$, $\phi_n\in D(\hat{B}(k))$ with $\|\phi_n\|=1$ such that
\bq\label{Bkn eigenequation}
\hat{B}(k_n)\phi_n=(L-i\hat{v}\cdot k_n)\phi_n=\lambda_n\phi_n,\quad {\rm Re}\lambda_n\rightarrow0,\,\,n\rightarrow\infty.
\eq
The above equation can be rewritten as
\bq
K\phi_n=(\lambda_n+\nu+i\hat{v}\cdot k_n)\phi_n.
\eq
Since $K$ is a compact operator on $L^2(\R_v^3)$, there exists a subsequence $\phi_{n_j}$ of $\phi_n$ and $g\in L^2(\R^3)$ such that
$$
K\phi_{n_j}\rightarrow g,\quad {\rm as} \quad j\rightarrow\infty.
$$
Since $|k_n|\in(\tau_0,\tau_1)$, $\sup_{|k|>\tau_0}|{\rm Im}\lambda_n|<\infty$, ${\rm Re}\lambda_n\rightarrow0$, there exists a subsequence $(k_{n_j},\lambda_{n_j})$ of $(k_n,\lambda_n)$  such that
$$
k_{n_j}\rightarrow k_0,\quad \lambda_n\rightarrow\lambda_0, \quad {\rm as} \quad j\rightarrow\infty.
$$
Hence we have
\bq
\lim_{j\rightarrow\infty}\phi_{n_j}=\lim_{j\rightarrow\infty}\frac{g}{\lambda_{n_j}+\nu+i\hat{v}\cdot k_{n_j}}=\frac{g}{\lambda_{0}+\nu+i\hat{v}\cdot k_{0}}=f_0.
\eq
Then when $n\rightarrow\infty$, the eigenvalue problem \eqref{Bkn eigenequation} is transformed to
\bq
\hat{B}(k_0)f_0=\lambda_0f_0.
\eq
 Thus $\lambda_0$ is an eigenvalue of $\hat{B}(k_0)$ with ${\rm Re}\lambda_0=0$, which contradicts the fact ${\rm Re}\lambda(k)<0$
for $|k|\ne 0$ established by Lemma~\ref{Specturm set}.
\end{proof}

\begin{thm}\label{Eigenvalues and eigenfunctions}
There exist a constant  $\tau_0>0$ such that the spectrum $ \sigma(\hat{B}(k))\cap \{\lambda\in \mathbb{C}\,|\,{\rm Re}\lambda\ge -\mu/2\}$  consists of five points $\{\lambda_j(|k|),\,j=-1,0,1,2,3\}$ for $|k|\le \tau_0$. The eigenvalues $\lambda_j(|k|)$ and the corresponding eigenfunctions $e_j=e_j(|k|,\omega)$ with $\omega=k/|k|$ are $C^\infty$ functions of $|k|$ for $|k|\le\tau_0$. In particular, the eigenvalues $\lambda_j(|k|)$ admit the following asymptotic expansion for $|k|\le\tau_0$,
\be\label{specr0-1}
\left\{\bln
&\lambda_{\pm1}(|k|)=\pm i\sqrt{a^2+b^2}|k|+A_{\pm1}|k|^2+O(|k|^3),\\
&\lambda_0(|k|)=A_{0}|k|^2+O(|k|^3),\\
&\lambda_2(|k|)=\lambda_3(|k|)=A_{2}|k|^2+O(|k|^3),
\eln\right.\ee
where $a,b>0$ and $A_j>0$, $j=-1,0,1,2$ are constants defined by
\bq \label{a}
\left\{\bal
a=(\hat{v}_1\psi_1,\psi_4)>0,\quad b=(\hat{v}_1\psi_0,\psi_1)>0,\\
A_{j}=- (L^{-1} P_1(\hat{v}\cdot\omega)E_j,(\hat{v}\cdot\omega)E_j)>0,
\\
E_{\pm1}=\sqrt{\frac{b^2}{2a^2+2b^2}}\psi_0\mp\sqrt{\frac{1}{2}}\omega\cdot\psi'+\sqrt{\frac{a^2}{2a^2+2b^2}}\psi_4,\\
E_0 =-\sqrt{\frac{a^2}{a^2+b^2}}\psi_0+\sqrt{\frac{b^2}{a^2+b^2}}\psi_4,\\
E_j =W_j\cdot\psi',\quad j=2,3,
\ea\right.
\eq
$\psi'=(\psi_1, \psi_2, \psi_3)$, and $W_j$, $j=2,3$ are orthonormal vectors satisfying $W_j\cdot\omega=0$.

The eigenfunctions $e_j=e_j(|k|,\omega)$ are orthogonal to each other and satisfy
\be\left\{\bal
(e_i(|k|,\omega),\overline{e_j(|k|,\omega)})=\delta_{ij},\quad i,j=-1,0,1,2,3,\\
 e_j(|k|,\omega)=e_{j,0}+e_{j,1}|k|+O(|k|^2),\quad |k|\le\tau_0,
\ea\right.\ee
where the coefficients $e_{j,n}$ are given as
\bq
  \left\{\bln                      \label{eigf1-1}
 &e_{j,0}=E_j(\omega),\quad j=-1,0,1,2,3,\\
 &e_{l,1}=\sum^1_{n=-1}b^l_{n}E_k+ i L^{-1} P_1(\hat{v}\cdot\omega)E_l,\quad l=-1,0,1, \\
 &e_{n,1}=  i L^{-1} P_1(\hat{v}\cdot\omega)E_n, \quad n=2,3,
  \eln\right.
  \eq  and $b^j_{n}$, $j,n=-1,0,1$ are defined by
  \be
   \left\{\bln
  &b^j_{j}=0,\quad
  b^j_{n}=\frac{(L^{-1}P_1(\hat{v}\cdot\omega)E_{j},(\hat{v}\cdot\omega)E_{n})}{i(u_n-u_j)},\,\,\ j\ne n;\\
  &u_{\pm1}=\mp\sqrt{a^2+b^2}, \quad u_0=0.
   \eln\right.
  \ee
\end{thm}
\begin{proof}
Since $L$ is invariant with respect to the rotation $\O$ of $v\in\R^3$, it follow that $\O \hat{B}(k)=\hat{B}(\O^{-1}k)$, which applied to $\hat{B}(k)e=\lambda e$ implies that the eigenvalue $\lambda$ depends only on $|k|$.
Now we consider eigenvalue problem in the form
\bq\label{characteristic equation1}
\hat{B}(k)e=|k|\beta e ,
\eq
that is
\bq\label{characteristic equation2}
(L-i|k|(\hat{v}\cdot\omega))e =|k|\beta e .
\eq
By macro-micro decomposition, the eigenfunction $e$ of \eqref{characteristic equation2} can be divided into
$$
e=P_0e+P_1e=g_0+g_1.
$$
Hence \eqref{characteristic equation2} gives
\bma
|k|\beta g_0=& -P_0[i|k|(\hat{v}\cdot\omega)(g_0+g_1)]\label{asymptotic expansion1},\\
|k|\beta g_1=& Lg_1-P_1[i|k|(\hat{v}\cdot\omega)(g_0+g_1)]\label{asymptotic expansion2}.
\ema
According to \eqref{asymptotic expansion2}, we have
\bq\label{asymptotic expansion3}
g_1=i(L-iP_1(|k|(\hat{v}\cdot\omega))-|k|\beta)^{-1}P_1(|k|(\hat{v}\cdot\omega))g_0.
\eq
By substitute \eqref{asymptotic expansion3} into \eqref{asymptotic expansion1}, we have
\bq\label{asymptotic expansion4}
|k|\beta g_0=-i|k|P_0(\hat{v}\cdot\omega)g_0+|k|P_0(\hat{v}\cdot\omega)(L-iP_1(|k|(\hat{v}\cdot\omega))-|k|\beta )^{-1}P_1(|k|(\hat{v}\cdot\omega))g_0.
\eq

Define the operator $A(\omega)=P_0(\hat{v}\cdot\omega)P_0$. We have the matrix representation of $A(\omega)$ as follows:
\be
\left(
\begin{matrix}
0& b\omega& 0\\
b\omega^T& 0& a\omega^T\\
0& a\omega& 0
\end{matrix}
\right),
\ee
where $\omega^T$ denotes the row vector $\omega's$ transpose and
$$
a=(\hat{v}_1\psi_1,\psi_4), \quad
b=(\hat{v}_1\psi_0,\psi_1) .
$$
It can be verified that the eigenvalues $u_i$ and normalized eigenvectors $E_i$ of $A$ are given by
\be\left\{\bal\label{asymptotic expansion5}
u_{\pm1}=\mp\sqrt{a^2+b^2},\quad u_j=0,\quad j=0,2,3,\\
 E_{\pm1}=\sqrt{\frac{b^2}{2a^2+2b^2}}\psi_0\mp\sqrt{\frac{1}{2}}\omega\cdot\psi'+\sqrt{\frac{a^2}{2a^2+2b^2}}\psi_4,\\
E_0=-\sqrt{\frac{a^2}{a^2+b^2}}\psi_0+\sqrt{\frac{b^2}{a^2+b^2}}\psi_4,\\
E_j=W_j\cdot\psi',\quad j=2,3,\\
 (E_i,E_j)=\delta_{ij},\quad -1\le i,j\le 3,
\ea\right.\ee
where $\psi'=(\psi_1, \psi_2, \psi_3)$ and $W_j$ are 3-dimensional normalized vectors such that
$$
W_2(\omega)\cdot W_3(\omega)=0,\quad W_2(\omega)\cdot\omega=W_3(\omega)\cdot\omega=0.
$$

To solve  the eigenvalue problem \eqref{asymptotic expansion4}, we write $\psi_0\in N_0$ in terms of the basis $E_{j}$  as
 \bq
 \psi_0=\sum_{j=0}^4C_jE_{j-1}\quad \text{with}\quad C_j=(\psi_0,E_{j-1}),\,\, j=0,1,2,3,4,\label{A_5a-1}
\eq
with the unknown coefficients $(C_0,C_1,C_2,C_3,C_4)$ to be determined below. Taking the inner product between \eqref{asymptotic expansion4} and $E_j$ for $j=-1,0,1,2,3$ respectively,
we have the equations about $\beta$ and $C_j$, $j=0,1,2,3,4$ for $\text{Re}\lambda>-\mu$:
 \be
 \beta C_j=-i u_{j-1}C_j +|k|\sum^4_{i=0}C_i D_{ij}(\beta,|k|,\omega),  \label{A_8-1}
 \ee
 where
 \be D_{ij}(\beta,|k|,\omega)=((L-i |k|P_1(\hat{v}\cdot\omega)-|k|\beta )^{-1}P_1 (\hat{v}\cdot\omega)E_{i-1},(\hat{v}\cdot\omega)E_{j-1}).\ee
Let $\O$ be a rotational transformation  in $\R^3$ such that
\be \O^T\omega= (1,0,0), \quad  \O^T W_2= (0,1,0), \quad  \O^T W_3= (0,0,1). \label{rot}\ee
By changing
variable $v\to \O v$, we have
\be
D_{ij}(\beta,|k|,\omega)=((L-i |k|P_1\hat{v}_1-|k|\beta )^{-1}P_1(\hat{v}_1F_{i-1}),\hat{v}_1F_{j-1})=:R_{ij}(\beta,|k|),\label{A_9-1}
\ee
where
\be \label{Fj}
\left\{\bal
 F_{\pm1}=\sqrt{\frac{b^2}{2a^2+2b^2}}\psi_0\mp\sqrt{\frac{1}{2}} \psi_1+\sqrt{\frac{a^2}{2a^2+2b^2}}\psi_4,\\
 F_0=-\sqrt{\frac{a^2}{a^2+b^2}}\psi_0+\sqrt{\frac{b^2}{a^2+b^2}}\psi_4,\\
 F_j= \psi_j,\quad j=2,3,\\
 (F_i,F_j)=\delta_{ij},\quad -1\le i,j\le 3.
\ea\right.
\ee
In particular, $R_{ij}(\beta,|k|)$, $i,j=0,1,2,3,4$ satisfy
\be \label{A_3-1}
\left\{\bln
&R_{ij}(\beta,|k|)=R_{ji}(\beta,|k|)=0,\quad i=0,1,2,\,\, j=3,4, \\
&R_{34}(\beta,|k|)=R_{43}(\beta,|k|)=0,\\
&R_{33}(\beta,|k|)=R_{44}(\beta,|k|).
\eln\right.
\ee
By \eqref{A_9-1} and \eqref{A_3-1}, we can divide \eqref{A_8-1} into two systems:
\bma
 \beta C_j&=-i u_{j-1}C_j +|k|\sum^2_{i=0}C_i R_{ij}(\beta,|k|),\quad j=0,1,2,  \label{A_6-1}\\
 \beta C_l&=-i u_{l-1}C_k +|k| C_l R_{33}(\beta,|k|),\quad l=3,4.  \label{A_7-1}
 \ema
Denote
\bma
D_0(\beta,|k|)&=\beta-|k|R_{33}(\beta,|k|),\label{BM-1}\\
D_1(\beta,|k|)&=\det\left|\bal
\beta+i u_{-1}-|k|R_{00} &  -|k|R_{10} &  -|k|R_{20} \\
-|k|R_{01} &  \beta+i u_0-|k|R_{11} &  -|k|R_{21} \\
-|k|R_{02} &-|k|R_{12} &  \beta+i u_1-|k|R_{22}
\ea\right|.\label{BM-1a}
\ema
The eigenvalues $\beta$ can be solved by $D_0(\beta,|k|)=0$ and $D_1(\beta,|k|)=0$. By a direct computation and the  implicit function theorem, we can show
\begin{lem}
\label{eigen_1}
 The equation $D_0(\beta,s)=0$ has a unique $C^\infty$ solution
$\beta=\beta(s)$ for $(s,\beta)\in[-\tau_0, \tau_0]\times B_{r_1}(0)$ with $\tau_0,r_1>0$ being small constants that satisfies
$$\beta(0)=0,\quad \beta'(0)=(L^{-1} P_1(\hat{v}_1F_2),\hat{v}_1F_2).$$
\end{lem}

We have the following result about  the solution of $D_1(\beta,|k|)=0$.
\begin{lem}\label{eigen_2}
There exist two small constants $\tau_0>0$ and $r_1>0$ so that the equation $D_1(\beta,s)=0$ admits three $C^\infty$ solutions $\beta_j(s)$ $(j=-1,0,1)$ for $(s,\beta_j)\in
[-\tau_0,\tau_0]\times B_{r_1}(-i u_j)$ that satisfy
 \be
 \beta_j(0)=-i u_j, \quad  \beta_j'(0) =(L^{-1} P_1(\hat{v}_1F_j),\hat{v}_1F_j).\label{T_5a-1}
 \ee
Moreover, $\beta_j(s)$  satisfies
 \bq
 -\beta_{j}(-s)=\overline{\beta_{j}(s)} =\beta_{-j}(s),\quad j=-1,0,1.\label{L_8-1}
 \eq
\end{lem}

\begin{proof}From \eqref{BM-1a},
 \bma
D_1(\beta,0)&=\det\left|\bal
\beta+i u_{-1} & 0 &  0 \\
0 &  \beta+i u_0 &  0 \\
0 & 0 &  \beta+i u_1
\ea\right|
\nnm\\
&=(\beta+i u_{-1})(\beta+i u_0)(\beta+i u_1).\label{A_13}
 \ema
 It follows that  $D_1(\beta,0)=0$ has three roots  $\beta_j=-i u_j$ for $j=-1,0,1$.
Since
  \bma
  {\partial_s}D_1(\beta,0)=&-\mathcal{D}_{-1}(\beta+i u_0)(\beta+i u_1)-\mathcal{D}_{0}(\beta+i u_{-1})(\beta+i u_1)\nnm\\
  &-\mathcal{D}_{1}(\beta+i u_{-1})(\beta+i u_0),     
\\
  {\partial_\beta}D_1(\beta,0)=&(\beta+i u_{-1})(\beta+i u_0)+(\beta+i u_{-1})(\beta+i u_1)\nnm\\
  &+(\beta+i u_0)(\beta+i u_1),  \label{lamd1}
 \ema
 where
 \be \mathcal{D}_{j}=(L^{-1}P_1(\hat{v}_1F_j),\hat{v}_1F_j)=A_j,\quad j=-1,0,1,\ee
it follows that
$${\partial_\beta}D_1(-i u_j,0)\ne 0.$$
The implicit function theorem implies that there exist
small constants $\tau_0,r_1>0$ and a unique $C^\infty$ function $\beta_j(s)$: $[-\tau_0,\tau_0]\to B_{r_1}(-i u_j)$ so that $D_1(\beta_j(s),s)=0$ for $s\in [-\tau_0,\tau_0]$, and in particular
\bq
\beta_j(0)=-i u_j,\quad \beta_j'(0)=-\frac{{\partial_s}D_1(-i u_j,0)}{{\partial_\beta}D_1(-i u_j,0)}=A_{j},
 \quad j=-1,0,1.                 \label{lamd2}
 \eq
 This proves \eqref{T_5a-1}.

Since
$ R_{ij}(\beta,-s)=R_{ij}(-\beta,s),  R_{ij}(\overline{\beta},s)=\overline{R_{jj}(\beta,s)}$ for $i,j=0,1,2$,
we obtain by \eqref{BM-1a}  that
$D_1(-\beta,s)=D_1(\beta,-s)$ and $\overline{D_1(\beta,s)}=D_1(\overline{\beta},s)$. This together with  the fact that $\beta_j(s)=-i u_j+O(s)$, $j=-1,0,1$ for $|s|\le \tau_0$ imply \eqref{L_8-1}.
\end{proof}

The eigenvalues $\lambda_j(|k|)$ and the eigenfunctions $e_j(|k|,\omega)$, $j=-1,0,1,2,3$ can be constructed as follows. For $j=2,3$, we take $\lambda_j=|k|\beta(|k|)$ to be the solution of the equation  $D_0(\beta,|k|)=0$ defined in Lemma \ref{eigen_1}, and choose  $C_i=0$, $i\ne j$. Thus the corresponding eigenfunctions $e_j(|k|,\omega)$, $j=2,3$ are defined by
 \bq
  e_j(|k|,\omega) =b_j(|k|)E_j(\omega) +i b_j(|k|)|k|[L-\lambda_j -i |k| P_1(\hat{v}\cdot\omega) ]^{-1}
         P_1(\hat{v}\cdot\omega)E_j(\omega),  \label{C_2-1}
\eq
which are orthonormal, i.e.,  $(e_2(|k|,\omega),\overline{e_3(|k|,\omega)}) =0$.

 For $j=-1,0,1$, we take $\lambda_j=|k|\beta_j(|k|)$ to be a solution of $D_1(\beta,|k|)=0$ given by Lemma \ref{eigen_2}, and choose $C_i=0$, $i=2,3$. Denote by $ \{C^j_0,\, C^j_1,\, C^j_2\}$  a solution of system \eqref{A_6-1} for $\beta=\beta_j(|k|)$. Then we can construct  $e_j(|k|,\omega)$, $j=-1,0,1$ as
 \bq
 \left\{\bln e_j(|k|,\omega)&= P_0e_j(|k|,\omega)+ P_1e_j(|k|,\omega),\\
 P_0e_j(|k|,\omega)&=C^j_0(|k|)E_{-1}(\omega)+C^j_1(|k|)E_0(\omega)+C^j_2(|k|)E_1(\omega),\\
 P_1e_j(|k|,\omega)&=i |k|[L-\lambda_j-i |k| P_1(\hat{v}\cdot\omega) ]^{-1} P_1[(\hat{v}\cdot\omega) P_0e_j(|k|,\omega)].
\eln\right.\label{C_3-1}
\eq

We write
$$(L-i |k|(\hat{v}\cdot\omega) )e_j(|k|,\omega)=\lambda_j(|k|)e_j(|k|,\omega), \quad -1\leq j\leq 3.$$
Taking the inner product $(\cdot,\cdot) $ of the above equation with $\overline{e_j(|k|,\omega)}$ and using the facts that
 \bgrs
 (\hat{B}(k) f,g) =(f,\hat{B}(-k)g) ,\quad f,g\in D(\hat{B}(k)),
\\
 \hat{B}(-k)\overline{e_j(|k|,\omega)} =\overline{\lambda_j(|k|)}\cdot\overline{e_j(|k|,\omega)},
 \egrs
we have
$$
(\lambda_j(|k|)-\lambda_{l}(|k|))(e_j(|k|,\omega),\overline{e_l(|k|,\omega)}) =0,\quad -1\le j, l\le 3.
$$
For $|k|\ne 0$ being sufficiently small, $\lambda_j(|k|)\neq\lambda_{l}(|k|)$ for
$-1\le j\neq l\le2$.  Therefore, we have
$$
(e_j(|k|,\omega),\overline{e_l(|k|,\omega)}) =0,\quad -1\leq j\neq l\leq3.
$$
We can normalize them by taking
$$(e_j(|k|,\omega),\overline{e_j(|k|,\omega)}) =1, \quad -1\leq j\leq 3.$$

The coefficients $b_j(|k|) $ for $j=2,3$ defined in \eqref{C_2-1} are determined by the normalization condition  as
 \be
 b_j(|k|)^2\(1-|k|^2D_j(|k|)\)=1,
 \eq
 where
 $$D_j(|k|)=((L-i |k|P_1\hat{v}_1-\lambda_j )^{-1} P_1\hat{v}_1F_j, (L+i |k|P_1\hat{v}_1-\overline{\lambda_j} )^{-1} P_1\hat{v}_1F_j).$$
 Substituting \eqref{specr0-1} into \eqref{C_2-1}, we obtain
 $$b_j(|k|)=1+\frac12|k|^2\|L^{-1}P_1\hat{v}_1F_j\|^2+O(|k|^3).$$
 This and \eqref{C_2-1} give the expansion of $e_j(|k|,\omega)$ for $j=2,3$, stated in \eqref{eigf1-1}.

To obtain expansion of $e_j(|k|,\omega)$ for $j=-1,0,1$ defined in
\eqref{C_3-1}, we deal with its macroscopic part and microscopic part respectively.
By \eqref{A_6-1}, the macroscopic part $ P_0e_j(|k|,\omega)$ is determined in terms of  the coefficients $ \{C^j_0(|k|),\, C^j_1(|k|),\, C^j_2(|k|)\}$ that satisfy
 \bq \label{expan2-1}
\beta_j(|k|) C^j_l(|k|)=-i u_{l-1}C^j_l(|k|)+|k|\sum^2_{n=0}C^j_n(|k|)R_{nl}(\beta_j,|k|),\quad l=0,1,2.
 \eq

 Furthermore, we have the normalization condition:
\be 1\equiv(e_j(|k|,\omega),\overline{e_j(|k|,\omega)}) =C^j_0(|k|)^2 +C^j_1(|k|)^2+C^j_2(|k|)^2+O(|k|^2),\quad  |k|\le \tau_0. \label{normal-1}\ee

Assume that
  $$
  C^j_l(|k|)= \sum_{n=0}^1  C^j_{l,n}|k|^n +O(s^2),\quad l=0,1,2,\,\, j=-1,0,1.
$$
Substituting the above expansion and \eqref{specr0-1} into \eqref{expan2-1} and \eqref{normal-1}, we have their expressions as
 \bma
O(1)&\qquad\qquad  \left\{\bal
-i u_jC^j_{l,0}=-i u_{l-1}C^j_{l,0}, \\
(C^j_{0,0})^2+(C^j_{1,0})^2+(C^j_{2,0})^2=1,
\ea\right. \label{m_4-1}
 \\
O(|k|) &\qquad\qquad  \left\{\bal
-i u_jC^j_{l,1}+ A_{j} C^j_{l,0}=-i u_{l-1}C^j_{l,1}+\sum^2_{n=0}C^j_{n,0}\mathcal{D}_{n-1,l-1},
 \\
C^j_{0,0}C^j_{0,1}+C^j_{1,0}C^j_{1,1}+C^j_{2,0}C^j_{2,1}=0,
 \ea\right.\label{C_6-1}
 \ema
 where $j=-1,0,1$, $l=0,1,2,$ and
 $$\mathcal{D}_{n,l}=(L^{-1}P_1(\hat{v}_1F_{n}),L^{-1}P_1(\hat{v}_1F_{l})).$$
By a direct computation, we obtain from \eqref{m_4-1}--\eqref{C_6-1} that
\be \label{expan3-1}
\left\{\bal
C^j_{j+1,0}=1,\quad C^j_{l,0}=1,\quad l\ne j+1,\\
C^j_{j+1,1}=0,\quad C^j_{l,1}=\frac{\mathcal{D}_{j,l-1}}{i(u_{l-1}-u_j)},\quad l\ne j+1.
\ea\right.
\ee
By \eqref{C_3-1} and \eqref{expan3-1},
we can obtain the expansion of $e_j(|k|,\omega)$ for $j=-1,0,1$ given in \eqref{eigf1-1}. The proof of this theorem is completed.
\end{proof}

\begin{remark}
(1) By changing variable $v\to \O v$ with $\O$ defined by \eqref{rot},  it is easy to verify that
\be  (L^{-1}P_1(\hat{v}\cdot\omega)E_j,(\hat{v}\cdot\omega)E_j)=(L^{-1}P_1( \hat{v}_1F_j),\hat{v}_1F_j),\ee
where $F_j$, $j=-1,0,1,2,3$ is given by \eqref{Fj}. Thus, the coefficients $A_j>0$ does not depend on $\omega$.

(2) Since $R_{ij}(\beta,s)$, $i,j=1,2,4$ are analytic in $(s,\beta)$, it follows that $D_0(\beta,s)$ and $D(\beta,s)$ are analytic in $(s,\beta)$. By implicit function theorem (section 8 of chapter 0 in \cite{homo}), the solution  $\beta(s)$ to $D_0(\beta,s)=0$   and the solutions $\beta_j(s)$, $j=-1,0,1$ to $D_1(\beta,s)=0$ are analytic functions of $s$ for $|s|\le \tau_0$. Thus,  the eigenvalues $\lambda_j(|k|)$, $j=-1,0,1,2,3$ of $\hat{B}(k)$ are analytic function of $|k|$ for $|k|\le \tau_0$.
\end{remark}
Theorem \ref{global spectrum} directly follows from Theorem \ref{high frequency spectrum} and Theorem \ref{Eigenvalues and eigenfunctions}.


\section{Optimal time-decay rates of linearized  equation}
\label{sect3}
\setcounter{equation}{0}
In this section, we will establish the optimal time-decay rates of global solution for Cauchy problem \eqref{lrb}.

\subsection{Decomposition and asymptotic behaviors of $e^{t\hat{B}(k)}$}

In this subsection, we decompose the semigroup $G(t,k)=e^{t\hat{B}(k)}$ and study asymptotic behaviors of this semigroup.
\begin{lem}\label{Alt}
For any $f\in L^2(\R_v^3)$ and $x>-c_0$, we have
$$
\int_{-\infty}^{+\infty}\|(x+iy-\hat{A}(k))^{-1}f\|^2dy\le\pi(x+c_0)^{-1}\|f\|^2.
$$
\end{lem}

\begin{proof}
We will use the similar calculation methods as \cite{Ukai1}. By using the Laplace transform, we have
\bq\label{Alt1}
(\lambda-\hat{A}(k))^{-1}=\int_{0}^{+\infty}e^{-\lambda t}e^{t\hat{A}(k)}dt,\quad {\rm Re}\lambda>-c_0,
\eq
which leads to
$$
(x+iy-\hat{A}(k))^{-1}=\frac1{\sqrt{2\pi}}\int_{-\infty}^{+\infty}e^{-iyt}\[ \sqrt{2\pi}1_{\{t\geq0\}}e^{-tx }e^{t\hat{A}(k)}\]dt,
$$
where the right hand side  is the Fourier transform of the function $\sqrt{2\pi}1_{\{t\geq0\}}e^{-xt}e^{tc(\xi)}$ with respect to $t$.
By Parseval's equality and Lemma \ref{spectrumA}, we have
\bmas
&\int_{-\infty}^{+\infty}\|(x+iy-\hat{A}(k))^{-1}f\|^2dy
=\int_{-\infty}^{+\infty}\|(2\pi)^{\frac{1}{2}}1_{\{t\geq0\}}e^{-xt }e^{t\hat{A}(k)}f\|^2dt\\
=&2\pi\int_{0}^{+\infty}e^{-2xt }\|e^{t\hat{A}(k)}f\|^2dt
\le 2\pi\int_{0}^{+\infty}e^{-2(x+c_0)t}dt\|f\|^2=\pi(x+c_0)^{-1}\|f\|^2.
\emas
This proves the lemma.
\end{proof}

\begin{lem}\label{IKAbounded}
Let  $x_0=-\mu/2 $ and $x_1= -d(\tau_0)$ with $d(\tau_0)$ defined in Theorem~\ref{high frequency spectrum}. Then, there exists a constant $C>0$ such that
\bq
\sup_{k\in\R^3,y\in\R}\|(I-K(-z+iy-\hat{A}(k))^{-1})^{-1}\|\le C,
\eq where $z=x_0$ for $|k|< \tau_0$ and $z=x_1$ for $|k|\ge \tau_0$.
\end{lem}

\begin{proof}
Let $\lambda=z+ i y$ with $z=x_0$ for $|k|< \tau_0$ and $z=x_1$ for $|k|\ge \tau_0$.
By Lemma \ref{KAbounded} and Theorem \ref{Eigenvalues and eigenfunctions}, we have $\lambda\in \rho(\Hat{B}(k))$, namely, $\lambda-\Hat{B}(k)$ is invertible. Thus
$$I-K(\lambda-\hat{A}(k))^{-1}=(\lambda-\Hat{B}(k))(\lambda-\hat{A}(k))^{-1}$$
is also invertible. By Lemma \ref{KAbounded},  there exist $R_0,R_1>0$ large enough such that for either $|k|\ge R_0$, or $|k|\le R_0$ and $|y|\ge R_1$,
\bq
\|K(\lambda-\hat{A}(k))^{-1}\|\le \frac{1}{2},
\eq
which yields
$$
\|(I-K(\lambda-\hat{A}(k))^{-1})^{-1}\|\le 2.
$$

It is sufficient to prove \eqref{IKAbounded} for $|y|\le R_1$ and $|k|\le R_0$. If it does not hold,  then there are sequences $\{k_n,\lambda_n=z+iy_n\}$ with $|k_n|\le R_0$, $|y_n|\le R_1$, and $\{f_n,g_n\}$ with $\|f_n\|\to0$,  $\|g_n\|=1$ such that
\bq\label{IKAbounded1}
g_n=(I-K(\lambda_n-\hat{A}(k_n))^{-1})^{-1}f_n .
\eq
This gives
\bq\label{w_n}
g_n-K(\lambda_n-\hat{A}(k_n))^{-1}g_n=f_n.
\eq
Let
$$w_n=(\lambda_n-\hat{A}(k_n))^{-1}g_n.$$
We can write \eqref{w_n} as
\bq\label{IKAbounded2}
(\lambda_n-\hat{A}(k_n))w_n-Kw_n=f_n.
\eq
Since
$$
\|w_n\|\le\|(\lambda-\hat{A}(k_n))^{-1}\| \|g_n\|\le C,
$$
 and $K$ is a compact operator  on $L^2(\R^3)$, there exists a
subsequence $w_{n_j}$ of $w_n$ and $h_0\in L^2(\R^3)$ such that
\bq\label{IKAbounded3}
Kw_{n_j}\to h_0, \quad \mbox{as}\quad j\to\infty.
\eq

Since $|k_n|\le R_0$, $|y_n|\le R_1$, there exists a subsequence of  (still denoted by)  $\{\lambda_{n_j}, k_{n_j}\}$  and $(\lambda_0, k_0) $ with $\lambda_0=z+iy_0$, $| k_0|\le R_0$, $|y_0|\le R_1$ such that
$$
\lambda_{n_j}\rightarrow \lambda_0,\quad  k_{n_j}\rightarrow  k_0,\quad \mbox{as}\quad j\rightarrow\infty.
$$

 Noting that $\lim_{n\rightarrow\infty}\|f_{n}\|=0$, we have by \eqref{IKAbounded2} and \eqref{IKAbounded3} that
$$
\lim_{j\rightarrow\infty}w_{n_j}=\lim_{j\rightarrow\infty}\frac{Kw_{n_j}+f_{n_j}}{\lambda_{n_j}+\nu(v)+i (\hat{v}\cdot k_{n_j})}=\frac{h_0}{\lambda_{0}+\nu(v)+i (\hat{v}\cdot k_{0})}=w_0,
$$
and hence $Kw_0=h_0$. Thus
$$
Kw_0=(\lambda_0+\nu(v)+i (\hat{v}\cdot k_{0}))w_0.
$$
It follows that $\lambda_0w_0=\hat{B}( k_0)w_0$ and  $\lambda_0$ is an eigenvalue of $\hat{B}( k_0)$ with ${\rm Re}\lambda_0=z$,  which contradicts the facts that ${\rm Re}\lambda(k)=\lambda_j(|k|)$, $j=-1,0,1,2,3$ for $|k|\le \tau_0$ and ${\rm Re}\lambda(k)< -d(\tau_0)$ for $|k|\ge \tau_0$.
\end{proof}

Then, we have the decomposition of the semigroup $G(t,k)=e^{t\hat{B}(k)}$ as below.
\begin{thm}\label{decomposition 0}
The semigroup $G(t,k)=e^{t\hat{B}(k)}$ with $k=|k|\omega$ satisfies
\bq\label{decomposition}
G(t,k)f=G_1(t,k)f+G_2(t,k)f,\quad f\in L^2(\R_v^3),
\eq
where
\bq
 G_1(t,k)=\sum_{j=-1}^3e^{\lambda_j(|k|)t}(f,\overline{e_j(|k|,\omega)})e_j(|k|,\omega)1_{\{|k|\le \tau_0\}},
 \eq
with $(\lambda_j(|k|),e_j(|k|,\omega))$ being the eigenvalue and eigenfunction of the operator $\hat{B}(k)$ given by Theorem~\ref{Eigenvalues and eigenfunctions} for $|\xi|\le \tau_0$, and $G_2(t,k)f =: G(t,k)f-G_1(t,k)f$ satisfy for a constant $\sigma_0>0$ independent of $k$,
 \bq
 \|G_2(t,k)f\|_{L_v^2}\le Ce^{-\sigma_0t}\|f\|_{L_v^2}.
\eq
\end{thm}
\begin{proof}The proof method is similar to \cite{Ukai1}. We can establish the asymptotic behavior of $G$
\bq\label{decomposition a}
G(t,k)f=e^{t\hat{B}(k)}f=\frac{1}{2\pi i}\int_{x-i\infty}^{x+i\infty}e^{\lambda t}(\lambda-\hat{B}(k))^{-1}fd\lambda,\quad x>0.
\eq
We can write the second resolvent equation for $\hat{A}(k)$ and $\hat{B}(k)$
\bq
(\lambda-\hat{B}(k))^{-1}=(\lambda-\hat{A}(k))^{-1}+(\lambda-\hat{B}(k))^{-1}K(\lambda-\hat{A}(k))^{-1}.
\eq
Combining this and \eqref{Reeigenvalue neg0 2}, we have
\bq\label{Reeigenvalue neg0 3}
(\lambda-\hat{B}(k))^{-1}=(\lambda-\hat{A}(k))^{-1}+Z(\lambda),
\eq
where
\bq
Z(\lambda)=(\lambda-\hat{A}(k))^{-1}(I-K(\lambda-\hat{A}(k))^{-1})^{-1}K(\lambda-\hat{A}(k))^{-1}.
\eq
Substitute this into \eqref{decomposition a}, we have
\bq\label{decomposition a1}
e^{t\hat{B}(k)}=e^{t\hat{A}(k)}+\frac{1}{2\pi i}\lim_{T\rightarrow\infty} U_{x,T},
\eq
and
\bq
U_{x,T}=\int_{-T}^Te^{(x+iy)t}Z(x+iy)dy,
\eq
where the positive constant $T>y_1$ and $y_1$ is defined by Theorem \ref{high frequency spectrum}. Set
$$\sigma_0=\frac{\mu}2,\,\,\, |k|< \tau_0;\quad \sigma_0=d(\tau_0) ,\,\,\, |k|\ge \tau_0.
$$
Since $Z(\lambda)$ is analytic in ${\rm Re}\lambda>-\sigma_0$ with only finite singularities at the eigenvalues $\lambda_j(|k|),\,j=-1,0,1,2,3$, we can shift the integration path from the ${\rm Re}\lambda=x>0$ to ${\rm Re}\lambda=-\sigma_0$ where $\sigma_0$ is given by Theorem \ref{Eigenvalues and eigenfunctions}. Then by the Residue Theorem, we have
\bq\label{decomposition 1}
U_{x,T}=H_T+2\pi i\sum_{j=-1}^3{\rm Res}\{e^{\lambda t}Z(\lambda)f;\lambda_j(|k|)\}+U_{-\sigma_0,T},
\eq
where Res$\{f(\lambda);\lambda_j\}$ means the residue of $f(\lambda)$ at $\lambda=\lambda_j$ and
$$
H_T=\frac{1}{2\pi i}\(\int_{-\sigma_0-iT}^{x-iT}-\int_{-\sigma_0+iT}^{x+iT}\)e^{\lambda t}Z(\lambda)d\lambda.
$$

We can estimate the right side of \eqref{decomposition 1} as follows. First, from Lemma \ref{KAbounded}, we verify that
\bq
\|H_T\|\rightarrow0,\quad T\rightarrow\infty.
\eq
Next, we have
\bq
{\rm Res}\{e^{\lambda t}Z(\lambda)f;\lambda_j(|k|)\}={\rm Res}\{e^{\lambda t}(\lambda-\hat{B}(k))^{-1};\lambda_j(|k|)\}=e^{\lambda_j(|k|)t}P_j(k)f1_{\{|k|\le \tau_0\}},
\eq
where
$$
P_j(k)f=(f,\overline{e_j(|k|,\omega)})e_j(|k|,\omega).
$$

Define
$$
U_{-\sigma_0,\infty}(t)=\lim_{T\rightarrow\infty}U_{-\sigma_0,T}(t)=\int_{-\infty}^{+\infty}e^{(-\sigma_0+iy)t}Z(-\sigma_0+iy)dy.
$$
By using Lemma \ref{Alt} and Lemma \ref{IKAbounded}, we have for any $f,g\in L^2(\R^3_v)$,
\bma
& |(U_{-\sigma_0,\infty}(t)f,g)|\nnm\\
\le & Ce^{-\sigma_0t}\int_{-\infty}^{+\infty}|(Z(-\sigma_0+iy)f,g)|dy\nnm\\
\le & C\|K\|e^{-\sigma_0t}\int_{-\infty}^{+\infty}\|(-\sigma_0+iy-\hat{A}(k))^{-1}f\|\|(-\sigma_0-iy-\hat{A}(k))^{-1}g\|dy\nnm\\
\le & C\|K\|e^{-\sigma_0t}(c_0-\sigma_0)^{-1}\|f\|\|g\|.
\ema
This implies that
\bq
\|U_{-\sigma_0,\infty}(t)\|\le Ce^{-\sigma_0 t}, \quad t\ge0.
\eq
Hence it follows from \eqref{Reeigenvalue neg0 3} and \eqref{decomposition a1}  that
\bq
G(t,k)f=G_1(t,k)f+G_2(t,k)f,
\eq
where
\bma
& G_1(t,k)f=\sum_{j=-1}^3e^{\lambda_{j}(k)t}P_j(k)f1_{\{|k|\le \tau_0\}}, \\
& G_2(t,k)f=e^{t\hat{A}(k)}f+U_{-\sigma_0,\infty}(t).
\ema
Moreover,
\bq
\|G_2(t,k)f\|\le\|e^{t\hat{A}(k)}f\|+\|U_{-\sigma_0,\infty}(t)\|\le Ce^{-\sigma_0t}\|f\|.
\eq
This proves the theorem.
\end{proof}

\subsection{Optimal time-decay rates of $e^{t\hat{B}(k)}$}
For any $f_0=L^2(\R^3_{x}\times \R^3_{v})$, set
$$e^{tB}f_0=(\F^{-1}e^{t\hat{B}(k)}\F)f_0.$$
By Lemma \ref{B strongly contraction semigroup}, we have
\bq
\|e^{tB}f_0\|_{H^N}^2=\intr(1+|k|^2)^N\|e^{t\hat{B}(k)}\hat{f}_0\|^2dk\le\intr(1+|k|^2)^N\|\hat{f}_0\|^2dk=\|f_0\|^2_{H^N}.
\eq
This implies that the linear operator $B$ generates a strongly continuous contraction semigroup $e^{tB}$ in $H^N$. Therefore $f(t,x,v)=e^{tB}f_0(x,v)$ is a global solution to \eqref{lrb} for any $f_0\in H^N$.

we have the time decay rates of the linearized relativistic Boltzmann equation \eqref{lrb} as follows.
\begin{thm}\label{Decay estimate upper limit}
Assume that $f_0\in H^N\cap Z_1$ for $N\ge0$. Then the global solution $f(t,x,v)=e^{tB}f_0(x,v)$ to the linearized relativistic Boltzmann equation \eqref{lrb} satisfies for any $\alpha,\alpha'\in\N^3$ with $\alpha'\le \alpha$ that
\bma\label{Decay estimate upper limit 1}
&\ \|(\dxa e^{tB}f_0,\psi_j)\|_{L_x^2}\le C(1+t)^{-(\frac{3}{4}+\frac{\zeta}{2}})(\|\dxa f_0\|_{L^2_{x,v}}+\|\partial_x^{\alpha'}f_0\|_{Z_1}),\quad j=0,1,2,3,4,\\
&\ \|P_1(\dxa e^{tB}f_0)\|_{L^2_{x,v}}\le C(1+t)^{-(\frac{5}{4}+\frac{\zeta}{2}})(\|\dxa f_0\|_{L^2_{x,v}}+\|\partial_x^{\alpha'}f_0\|_{Z_1}),\label{Decay estimate upper limit 2}
\ema
where $\zeta=|\alpha-\alpha'|\le N$. If it further holds that $P_0f_0=0$, then
\bma\label{Decay estimate upper limit 3}
&\ \|(\dxa e^{tB}f_0,\psi_j)\|_{L_x^2}\le C(1+t)^{-(\frac{5}{4}+\frac{\zeta}{2}})(\|\dxa f_0\|_{L^2_{x,v}}+\|\partial_x^{\alpha'}f_0\|_{Z_1}),\quad j=0,1,2,3,4,\\
&\ \|P_1(\dxa e^{tB}f_0)\|_{L^2_{x,v}}\le C(1+t)^{-(\frac{7}{4}+\frac{\zeta}{2}})(\|\dxa f_0\|_{L^2_{x,v}}+\|\partial_x^{\alpha'}f_0\|_{Z_1}).\label{Decay estimate upper limit 4}
\ema
\end{thm}
\begin{proof}
By theorem \ref{decomposition 0} and the Planchel's equation,
\bma\label{Decay estimate upper limit 5}
\|(\dxa e^{tB}f_0,\psi_j)\|_{L_x^2}&=\|k^\alpha(G(t,k)\hat{f}_0,\psi_j)\|_{L_k^2}\nnm\\
&\le \|k^\alpha(G_1(t,k)\hat{f}_0,\psi_j)\|_{L_k^2}+\|k^\alpha(G_2(t,k)\hat{f}_0,\psi_j)\|_{L_k^2}.
\ema
We can estimate the second terms on the right hand of \eqref{Decay estimate upper limit 5} as follows:
\bma\label{Decay estimate upper limit 6}
\|k^\alpha(G_2(t,k)\hat{f}_0,\psi_j)\|_{L_k^2}^2 &\le \intr(k^\alpha)^2\|G_2(t,k)\hat{f}_0\|_{L_{v}^2}^2dk\nnm\\
&\le C\intr e^{-2\sigma_0t}(k^\alpha)^2\|\hat{f}_0\|_{L_v^2}^2dk\le Ce^{-2\sigma_0t}\|\dxa\hat{f}_0\|^2_{L_{x,v}^2}.
\ema
Next, we establish estimation the first term in the right hand side of \eqref{Decay estimate upper limit 5}. By \eqref{decomposition}, we have for $|k|\le \tau_0$,
\bq
G_1(t,k)\hat{f}_0=\sum_{j=-1}^3 e^{t\lambda_j(|k|)}[(\hat{f}_0,g_0)g_0+|k|T_j(k)\hat{f}_0],
\eq
where $T_j(k),-1\le j \le 3$ is the linear operator with the norm $\|T_j(k)\|$ being uniformly bounded for  $|k|\le \tau_0$. Therefore, we have
\bma\label{Decay estimate upper limit A}
(G_1(t,k)\hat{f}_0,\psi_0)=
&\ \frac{1}{2}B_1\sum_{j=\pm1}e^{t\lambda_j(|k|)}(B_1\hat{n}_0-j(\hat{m}_0\cdot\omega)+B_2\hat{q}_0)
\nnm\\
&\ -B_2e^{t\lambda_0(|k|)}( -B_2\hat{n}_0+B_1\hat{q}_0)+|k|\sum_{j=-1}^3e^{t\lambda_j(|k|)}(T_j(k)\hat{f}_0,\psi_0),
\\
\label{Decay estimate upper limit A1}
(G_1(t,k)\hat{f}_0,\psi')=&\ -\frac{1}{2}\sum_{j=\pm1}e^{t\lambda_j(|k|)}j(B_1\hat{n}_0 -j(\hat{m}_0\cdot\omega)+B_2\hat{q}_0)\omega
\nnm\\
&\ +\sum_{j=2,3}e^{t\lambda_j(|k|)}(\hat{m}_0\cdot W_j)W_j+|k|\sum_{j=-1}^3e^{t\lambda_j(|k|)}(T_j(k)\hat{f}_0,\psi'),
\\
\label{Decay estimate upper limit A2}
(G_1(t,k)\hat{f}_0,\psi_4)=
&\frac{1}{2}B_2\sum_{j=\pm1}e^{t\lambda_j(|k|)}(B_1\hat{n}_0-j(\hat{m}_0\cdot\omega)+B_1\hat{q}_0)
\nnm\\
&\ +B_1e^{t\lambda_0(|k|)}( B_1\hat{q}_0-B_2\hat{n}_0)+|k|\sum_{j=-1}^3e^{t\lambda_j(|k|)}(T_j(k)\hat{f}_0,\psi_4),
\ema
where $B_1=\sqrt{\frac{b^2}{a^2+b^2}}$, $B_2=\sqrt{\frac{a^2}{a^2+b^2}}$, $(\hat{n}_0, \hat{m}_0, \hat{q}_0)=((\hat f_0,\psi_0),(\hat f_0,\psi'),(\hat f_0,\psi_4))$ is the Fourier transform of the macroscopic density, momentum and energy of the initial data $f_0$, $W_j$ is given by theorem \ref{Eigenvalues and eigenfunctions} and
\bq\label{Decay estimate upper limit 7}
P_1(G_1(t,k)\hat{f}_0)=|k|\sum_{j=-1}^3e^{t\lambda_j(|k|)}P_1(T_j(k)\hat{f}_0).
\eq
Since
\bq\label{Decay estimate upper limit 8}
\mathrm{Re}\lambda_j(|k|)=A_{j}|k|^2(1+O(|k|))\le -\rho|k|^2, \quad  |k|\le \tau_0,
\eq
we have
\bma\label{Decay estimate upper limit a}
 \|k^\alpha(G_1(t,k)\hat{f}_0,\psi_j)\|_{L_k^2}^2
&\ \le C\int_{|k|\le \tau_0}(k^{\alpha-\alpha'})^2e^{-2\rho|k|^2t}(|k^{\alpha'}(\hat{n}_0,\hat{m}_0,\hat{q}_0)|^2 +|k|^2\|k^{\alpha'}\hat{f}_0\|^2_{L_v^2})dk\nnm\\
&\ \le C(1+t)^{-(\frac{3}{2}+\zeta)}(\|\partial_x^{\alpha'}(n_0,m_0,q_0)\|^2_{L^1_x} +|\partial_x^{\alpha'}f_0\|^2_{Z_1}) \nnm\\
&\ \le C(1+t)^{-(\frac{3}{2}+\zeta)}\|\partial_x^{\alpha'}f_0\|^2_{Z_1},
\ema
where $j=-1,0,1,2,3$.
Combining \eqref{Decay estimate upper limit 5}, \eqref{Decay estimate upper limit 6} and \eqref{Decay estimate upper limit a}, we verify that \eqref{Decay estimate upper limit 1} holds.

For $P_1(\dxa e^{tB}f_0)$, we have estimation that
\bq\label{Decay estimate upper limit P1}
\|P_1(\dxa e^{tB}f_0)\|^2_{L^2_{x,v}}\le\|k^\alpha P_1(G_1(t,k)\hat{f}_0)\|^2_{L^2_{k,v}}+\|k^\alpha P_1(G_2(t,k)\hat{f}_0)\|^2_{L^2_{k,v}}.
\eq
By \eqref{Decay estimate upper limit 7} and \eqref{Decay estimate upper limit 8}, we have
\bma\label{Decay estimate upper limit 11}
\|k^\alpha P_1(G_1(t,k)\hat{f}_0)\|^2_{L^2_{k,v}}&\ \le C\int_{|k|\le \tau_0}(k^{\alpha-\alpha'})^2e^{-2\rho|k|^2t}\|k^{\alpha'}\hat{f}_0\|^2_{L^2_v}dk\nnm\\
&\ \le C(1+t)^{-(\frac{5}{2}+\zeta)}\|\partial_x^{\alpha'}f_0\|_{Z_1},
\ema
Combing \eqref{Decay estimate upper limit 5}, \eqref{Decay estimate upper limit 6}, \eqref{Decay estimate upper limit P1} and \eqref{Decay estimate upper limit 11}, we verify that \eqref{Decay estimate upper limit 2} holds.

For the case $P_0f_0=0$, we have
\bq
 G_1(t,k)\hat{f}_0=|k|\sum_{j=-1}^3e^{\lambda_j(|k|)t} (\hat{f}_0,P_1(\overline{e_{j,1}}))e_{j,0}+|k|^2T_4(t,k)\hat{f}_0 ,\quad |k|\le \tau_0,
 \eq
 where $T_4(t,k)$ is also a linear operator satisfying for $|k|\le \tau_0$ that
 \bq
 \|T_4(t,k)\hat{f}\|_{L^2_v}\le Ce^{-\rho|k|^2t}\|\hat{f}\|_{L^2_v}.
 \eq
Then we have
after a direct computation that
\bma
(G_1(t,k)\hat{f}_0,\psi_0)
&=i\frac12B_1|k|\sum_{j=\pm1}e^{\lambda_j(|k|)t}(B_1\eta-j(\theta\cdot\omega)+B_2 \gamma)
\nnm\\
&\quad-i B_2 |k|e^{\lambda_0(|k|)t}(-B_2\eta+B_1 \gamma)+|k|^2 (T_4(t,k)\hat{f}_0,\psi_0),\label{V_5-1}\\
(G_1(t,k)\hat{f}_0,\psi')
&=-i\frac{1}{2}|k|\sum_{j=\pm1}e^{\lambda_j(|k|)t}j(B_1\eta-j(\theta\cdot\omega)+B_2 \gamma)\omega
\nnm\\
&\quad+ i|k|\sum_{j=2,3}e^{\lambda_j(|k|)t}(\theta\cdot W_j)W_j+|k|^2 (T_4(t,k)\hat{f}_0,\psi'),\label{V_6-1}\\
(G_1(t,k)\hat{f}_0,\psi_4)
&=i\frac12B_2|k|\sum_{j=\pm1}e^{\lambda_j(|k|)t}(B_1\eta-j(\theta\cdot\omega)+B_2 \gamma)
\nnm\\
&\quad+i B_1 |k|e^{\lambda_0(|k|)t}(-B_2\eta+B_1 \gamma)+|k|^2 (T_4(t,k)\hat{f}_0,\psi_4),\label{V_7-1}
\ema
where $\eta,\theta$ and $\gamma$ are defined by
\bq
\eta=(\hat f_0,L^{-1} P_1(\hat{v}\cdot\omega)\psi_0),\quad \theta=(\hat f_0,L^{-1} P_1(\hat{v}\cdot\omega)\psi'),\quad \gamma=(\hat f_0,L^{-1} P_1(\hat{v}\cdot\omega)\psi_4),\label{ab}
\eq
and
\bq  P_1(G_1(t,k)\hat f_0)=|k|^2 P_1(T_4(t,k)\hat f_0).\label{V_8-1}\eq

 Similarly, repeating the above steps, we can establish time-decay estimation \eqref{Decay estimate upper limit 3}--\eqref{Decay estimate upper limit 4}.
\end{proof}

We can also prove that the above time-decay rates are indeed optimal in following sense.

\begin{thm}\label{Decay estimate lower limit}Assume that $f_0\in H^N\cap Z_1$ for $N\ge0$ and there exist positive constants $d_0$, $d_1$ such that $\hat{f}_0$ satisfies that $\inf_{k\le \tau_0}|(\hat{f_0},\psi_0)|\ge d_0$, $ \sup_{|k|\le \tau_0}|(\hat{f}_0,b\psi_4-a\psi_0)|=0$ and $\sup_{k\le \tau_0}|(\hat{f}_0,\psi')|=0$ with $\psi'=(\psi_1,\psi_2,\psi_3)$. Then the global solution $f(t,x,v)=e^{tB}f_0(x,v)$ to the linear relativistic Boltzmann equation \eqref{lrb} satisfies
\bma\label{Decay estimate lower limit 1}
&\ C_1(1+t)^{-\frac{3}{4}}\le\|(e^{tB}f_0,\psi_j)\|_{L_x^2}\le C_2(1+t)^{-\frac{3}{4}},\quad j=0, 4,\\
\label{Decay estimate lower limit 1a}&\ C_1(1+t)^{-\frac{3}{4}}\le\|(e^{tB}f_0,\psi')\|_{L_x^2}\le C_2(1+t)^{-\frac{3}{4}}, \\
\label{Decay estimate lower limit 2}&\ C_1(1+t)^{-\frac{5}{4}}\le\|P_1(e^{tB}f_0)\|_{L_{x,v}^2}\le C_2(1+t)^{-\frac{5}{4}},
\ema
for $t>0$  large and two positive constants $C_2\ge C_1$.

If it  holds that $P_0f_0=0$, $\inf_{k\le \tau_0}|\hat{f}_0,L^{-1}P_1(\hat{v}\cdot\omega)(\omega\cdot\psi')|\ge d_0$ and $\sup_{k\le \tau_0}|(\hat{f}_0,L^{-1}P_1(\hat{v}\cdot\omega)\psi_j)|=0$ for $j=0,4$ and $\omega=k/|k|$, then
\bma\label{Decay estimate lower limit 4}
&\ C_1(1+t)^{-\frac{5}{4}}\le\|(e^{tB}f_0,\psi_j)\|_{L_x^2}\le C_2(1+t)^{-\frac{5}{4}},\quad j=0, 4,\\
\label{Decay estimate lower limit 4a}&\ C_1(1+t)^{-\frac{5}{4}}\le\|(e^{tB}f_0,\psi')\|_{L_x^2}\le C_2(1+t)^{-\frac{5}{4}}, \\
\label{Decay estimate lower limit 5}&\ C_1(1+t)^{-\frac{7}{4}}\le\|P_1(e^{tB}f_0)\|_{L_{x,v}^2}\le C_2(1+t)^{-\frac{7}{4}},
\ema
for $t>0$  large.
\end{thm}
\begin{proof}
By \eqref{Decay estimate upper limit 6}, we have
\bma\label{Decay estimate lower limit 7}
\|(e^{tB}f_0,\psi_j)\|_{L_x^2}&\ \ge\|(G_1(t,k)\hat{f}_0,\psi_j)\|_{L^2_k}-\|G_2(t,k)\hat{f}_0\|_{L^2_{k,v}}\nnm\\
&\ \ge\|(G_1(t,k)\hat{f}_0,\psi_j)\|_{L^2_k}-Ce^{-\sigma_0t}\|f_0\|_{L^2_{x,v}}.
\ema
By \eqref{Decay estimate upper limit A} and
\be b\hat q_0=a\hat n_0 ,\quad \hat m_0=0, \quad \lambda_{-1}(|k|)=\overline{\lambda_1(|k|)}, \quad {\rm for } \quad |k|\le \tau_0, \label{con}\ee
we obtain
\bma
|(G_1(t,k)\hat{f}_0,\psi_0)|^2
= &\bigg | \frac{b}{a^2+b^2} (b\hat{n}_0+a\hat{q}_0) e^{{\rm Re}\lambda_1(|k|)t}\cos({\rm Im}\lambda_1(|k|)t)\nnm\\
& +|k|\sum_{j=-1}^3e^{t\lambda_j(|k|)}(T_j(k)\hat{f}_0,\psi_0)\bigg|^2\nnm\\
\ge & \frac{1}{2} |\hat{n}_0|^2 e^{2{\rm Re}\lambda_1(|k|)t}\cos^2({\rm Im}\lambda_1(|k|)t
 -Ce^{-2\rho|k|^2t}|k|^2\|\hat{f}_0\|^2_{L^2_v}.
\ema
Since
\bq
{\rm Re}\lambda_j(|k|)=\lambda_{j,2}|k|^2(1+O(|k|))\ge -\xi|k|^2, \quad |k|\le \tau_0,
\eq
and
\bq
\cos^2({\rm Im}\lambda_1(|k|)t)\ge \frac{1}{2}\cos^2[\sqrt{a^2+b^2}|k|t]-O([|k|^3t]^2),
\eq
it follows that
\bma\label{Decay estimate lower limit 8}
\|(G_1(t,k)\hat{f}_0,\psi_0)\|^2_{L^2_v}\ge
&\ \frac{1}{2}d_0^2\int_{|k|\le \tau_0}\frac{1}{2}e^{-2\xi|k|^2t}\cos^2(\sqrt{a^2+b^2}|k|t)dk\nnm\\
&\ -C\int_{|k|\le \tau_0}e^{-2\rho|k|^2t}(|k|^6t^2|\hat n_0|^2+|k|^2\|\hat{f}_0\|^2_{L^2_v})dk\nnm\\
\ge &\ C\int_{|k|\le \tau_0}e^{-2\xi|k|^2t}\cos^2(\sqrt{a^2+b^2}|k|t)dk -C(1+t)^{-\frac{5}{2}}.
\ema
Set $$I_1=\int_{|k|\le \tau_0}e^{-2\xi|k|^2t}\cos^2(\sqrt{a^2+b^2}|k|t)dk.$$

we can estimate for $t\ge t_0=\frac{L^2}{\tau_0^2}$ with the constant $L\ge\sqrt{\frac{4\pi}{b_1}}$ that
\bma\label{Decay estimate lower limit 9}
I_1=&\ t^{-\frac{3}{2}}\int_{|s|\le \tau_0\sqrt{t}}e^{-2\xi|s|^2}\cos^2(\sqrt{a^2+b^2}|s|\sqrt{t})ds\nnm\\
\ge &\ 4\pi(1+t)^{-\frac{3}{2}}\int_0^1e^{-2\xi r^2}\cdot r^2\cos^2(\sqrt{a^2+b^2}r\sqrt{t})dr\nnm\\
\ge &\ \pi(1+t)^{-\frac{3}{2}}L^2e^{-2\xi L^2}\int_{\frac{L}{2}}^L\cos^2(\sqrt{a^2+b^2}r\sqrt{t})dr\nnm\\
\ge &\ \pi(1+t)^{-\frac{3}{2}}L^2e^{-2\xi L^2}\int_0^{\pi}\cos^2ydy\nnm\\
\ge &\ C(1+t)^{-\frac{3}{2}}.
\ema
By substitute \eqref{Decay estimate lower limit 8} and \eqref{Decay estimate lower limit 9} into \eqref{Decay estimate lower limit 7}, we verify that \eqref{Decay estimate lower limit 1} holds for $j=0$.

By \eqref{Decay estimate upper limit A1} and \eqref{con}, we can calculate
\bma
|(G_1(t,k)\hat{f}_0,\psi')|^2
=&\
\bigg|i(B_1\hat{n}_0+B_2\hat{q}_0)\omega\cdot e^{{\rm Re}\lambda_1(|k|)t}\sin({\rm Im}\lambda_{1}(|k|)t))\nnm\\
&\ +|k|\sum_{j=-1}^3e^{t\lambda_j(|k|)}(T_j(k)\hat{f}_0,\psi')\bigg|^2\nnm\\
\ge &\
\frac{1}{2B_1^2} |\hat{n}_0|^2  e^{2{\rm Re}\lambda_1(|k|)t}\sin^2({\rm Im}\lambda_{1}(|k|)t) -C|k|^2 e^{-2\rho|k|^2t}\|\hat{f}_0\|^2_{L^2_v}.
\ema
In terms of the fact that
\bq
\sin^2({\rm Im}\lambda_1(|k|)t)\ge\frac{1}{2}\sin^2(\sqrt{a^2+b^2}|k|t)-O([|k|^3t]^2),
\eq
we have
\bma\label{Decay estimate lower limit 10}
\|(G_1(t,k)\hat{f}_0,\psi')\|^2_{L^2_k}\ge &\
\frac{1}{2B_1^2}d_0^2\int_{|k|\le \tau_0}e^{-2\xi|k|^2t}\sin^2(\sqrt{a^2+b^2}|k|t)dk\nnm\\
&\ -C\int_{|k|\le \tau_0}e^{-2\rho|k|^2t}(|k|^6t^2|\hat n_0|^2+|k|^2\|\hat{f}_0\|^2_{L^2_v})dk\nnm\\
\ge &\
\frac{1}{2B_1^2}d_0^2\int_{|k|\le \tau_0}e^{-2\xi|k|^2t}\sin^2(\sqrt{a^2+b^2}|k|t)dk  -C(1+t)^{-\frac{5}{2}}.
\ema
Set $$I_2=\int_{|k|\le \tau_0}e^{-2\xi|k|^2t}\sin^2(\sqrt{a^2+b^2}|k|t)dk.$$
It holds for $t\ge t_0=\frac{L^2}{\tau_0^2}$ with the constant $L>4\pi$ that
\bma\label{Decay estimate lower limit 11}
I_2=&\ t^{-\frac{3}{2}}\int_{|s|\le \tau_0\sqrt{t}}e^{-2\xi|s|^2}\sin^2(\sqrt{a^2+b^2}|s|\sqrt{t})ds\nnm\\
\ge &\ 4\pi(1+t)^{-\frac{3}{2}}\int_0^1e^{-2\xi r^2}\cdot r^2\sin^2(\sqrt{a^2+b^2}r\sqrt{t})dr\nnm\\
\ge &\ \pi(1+t)^{-\frac{3}{2}}L^2e^{-2\xi L^2}\int_{\frac{L}{2}}^L\sin^2(\sqrt{a^2+b^2}r\sqrt{t})dr\nnm\\
\ge &\ \pi(1+t)^{-\frac{3}{2}}L^2e^{-2\xi L^2}\int_0^{\pi}\sin^2ydy\nnm\\
\ge &\ C(1+t)^{-\frac{3}{2}}.
\ema
Substitute \eqref{Decay estimate lower limit 10} and \eqref{Decay estimate lower limit 11} into \eqref{Decay estimate lower limit 7}, we prove \eqref{Decay estimate lower limit 1a}.

By \eqref{Decay estimate upper limit A2} and \eqref{con}, we can calculate
\bma
|(G_1(t,k),\psi_4)|^2= &\ \bigg| \frac{a}{a^2+b^2}(b\hat{n}_0+a\hat{q}_0) e^{{\rm Re}\lambda_1(|k|)t}\cos({\rm Im}\lambda_1(|k|)t)) \nnm\\
&\ +|k|\sum_{j=-1}^3e^{t\lambda_j(|k|)}(T_j(k)\hat{f}_0,\psi_4)\bigg|^2.
\ema
Then, we can get the same conclusion as the estimate term \eqref{Decay estimate lower limit 8} that
\bq\label{Decay estimate lower limit 12}
\|(G_1(t,k),\psi_4)\|^2_{L^2_k}\ge C_1(1+t)^{-\frac{3}{2}} -C_2(1+t)^{-\frac{5}{2}}.
\eq
and the proof process is the same as \eqref{Decay estimate lower limit 1} for $j=0$. Thus, by substitute \eqref{Decay estimate lower limit 12} into \eqref{Decay estimate lower limit 7}, we verify \eqref{Decay estimate lower limit 1} holds for $j=4$. In conclusion, \eqref{Decay estimate lower limit 1} holds for $t>0$ large.

Next, we prove \eqref{Decay estimate lower limit 2}. By \eqref{Decay estimate upper limit 6}, we have
\bma\label{Decay estimate lower limit 13}
\|P_1(e^{tB}f_0)\|_{L^2_{x,v}}\ge &\ \|P_1(G_1(t,k)\hat{f}_0)\|_{L^2_{k,v}}-\|P_1(G_2(t,k)\hat{f}_0)\|_{L^2_{k,v}}\nnm\\
\ge &\ \|P_1(G_1(t,k)\hat{f}_0)\|_{L^2_{k,v}}-Ce^{-\sigma_0t}\|f_0\|_{L^2_{x,v}}.
\ema
By \eqref{eigf1-1} and \eqref{Decay estimate upper limit 7}, we have
\bma
&\ P_1(G_1(t,k)\hat{f}_0)\nnm\\
=&\ |k|\sum_{j=-1}^3e^{\lambda_j(|k|)t}(\hat{f}_0,e_{j,0}) P_1(e_{j,1})+|k|^2T_5(t,k)\hat{f}_0\nnm\\
=&\ -i\sqrt{\frac12}|k|\sum_{j=\pm 1} e^{t\lambda_{j}(|k|)}( B_1\hat{n}_0+B_2\hat{q}_0) L^{-1}P_1(\hat{v}\cdot\omega)e_{j,0 }  +|k|^2T_5(t,k)\hat{f}_0\nnm\\
= &\ - i|k| B_1^{-1}\hat{n}_0 e^{ {\rm Re}\lambda_{1}(|k|)t}\cos({\rm Im}\lambda_1(k)t) L^{-1}P_1(\hat{v}\cdot\omega)(B_1\psi_0+B_2\psi_4) \nnm\\
&\  - i B_1^{-1}\hat{n}_0e^{ {\rm Re}\lambda_{1}(|k|)t}\sin({\rm Im}\lambda_1(k)t) L^{-1}P_1(\hat{v}\cdot\omega)(\omega\cdot\psi') +|k|^2T_5(t,k)\hat{f}_0,
\ema
where $T_5(t,k)$ satisfies
\bq
\|T_5(t,k)\hat{f}_0\|^2_{L^2_v}\le Ce^{-2\rho|k|^2t}\|\hat{f}_0\|^2_{L^2_v}.
\eq
Therefore,
\bma
\|P_1(G_1(t,k)\hat{f}_0)\|^2_{L^2_v}\ge &\ \frac{1}{2B_1^2}|k|^2|\hat{n}_0|^2\| L^{-1}P_1 \hat{v}_1(B_1\psi_0+B_2\psi_4)\|^2_{L^2_v}  e^{ 2{\rm Re}\lambda_{1}(|k|)t}\cos^2({\rm Im}\lambda_1(|k|)t)\nnm\\
&\ +\frac{1}{2B_1^2}|k|^2 |\hat{n}_0|^2\| L^{-1}P_1 \hat{v}_1\psi_1 \|^2_{L^2_v} e^{ {\rm Re}\lambda_{1}(|k|)t}\sin^2({\rm Im}\lambda_1(|k|)t)\nnm\\
&\ -C|k|^4e^{-2\rho|k|^2t}\|\hat{f}_0\|^2_{L^2_v}.
\ema
Further, by direct calculation, we have
\bma\label{Decay estimate lower limit 14}
&\ \|P_1(G_1(t,k)\hat{f}_0)\|^2_{L^2_{k,v}}\nnm\\
\ge &\
\frac{1}{2B_1^2}d_0^2\| L^{-1}P_1 \hat{v}_1\psi_1 \|^2_{L^2_v}\int_{|k|\le \tau_0}|k|^2e^{-2\xi|k|^2t}\sin^2[\sqrt{a^2+b^2}t]dk\nnm\\
&\  -C\int_{|k|\le \tau_0}e^{-2\rho|k|^2t}(|k|^4t)^2|\hat{n}_0|^2dk-C\int_{|k|\le \tau_0}|k|^4e^{-2\rho|k|^2t}\|\hat{f}_0\|^2_{L^2_v}dk\nnm\\
\ge &\ C(1+t)^{-\frac{5}{2}}-C(1+t)^{-\frac{7}{2}}.
\ema
By substitute \eqref{Decay estimate lower limit 14} into \eqref{Decay estimate lower limit 13}, we verify \eqref{Decay estimate lower limit 2} holds for $t>0$ large.

For the case $P_0f_0=0$, when $|k|\le k_0$ and $(f_0,L^{-1}P_1(\hat{v}\cdot\omega)\psi_j)=0$ for $j=0,4$, we have
\bma
(G_1(t,k)\hat{f}_0,\psi_0)=&\ -\frac{1}{2}B_1i|k| \sum_{j=\pm1}e^{t\lambda_j(|k|)}jh+|k|^2(T_4(t,k)\hat{f}_0,\psi_0),\\
(G_1(t,k)\hat{f}_0,\psi')=&\ \frac{1}{2}i|k| \sum_{j=\pm1}e^{t\lambda_j(|k|)}jh\omega+i|k|\sum_{j=2,3}e^{t\lambda_j(|k|)}l_jW_j+|k|^2(T_4(t,k)\hat{f}_0,\psi'),\\
(G_1(t,k)\hat{f}_0,\psi_4)=&\ -\frac{1}{2}B_2i|k| \sum_{j=\pm1}e^{t\lambda_j(|k|)}jh+|k|^2(T_4(t,k)\hat{f}_0,\psi_4),\\
P_1(G_1(t,k)\hat{f}_0)=&\ \frac{1}{2}|k|^2\sum_{j=\pm1}e^{t\lambda_j(|k|)}h (-jL^{-1}P_1(\hat{v}\cdot\omega)(\omega\cdot\psi')+
 L^{-1}P_1(\hat{v}\cdot\omega)(B_1\psi_0+B_2\psi_4))\nnm\\
&\ +|k|^2\sum_{j=2,3}e^{t\lambda_j(|k|)}l_j  L^{-1}P_1(\hat{v}\cdot\omega)(W_j\cdot\psi')+|k|^3T_6(t,k)\hat{f}_0,
\ema
where $h=(f_0,L^{-1}P_1(\hat{v}\cdot\omega)(\omega\cdot\psi'))$, $l_j=(f_0,L^{-1}P_1(\hat{v}\cdot\omega)(W_j\cdot\psi'))$, $W_j$ is given by theorem \ref{Eigenvalues and eigenfunctions} and $T_6(t,k)$ satisfies
\bq
\|T_6(t,k)\hat{f}_0\|^2_{L^2_v}\le Ce^{-2\rho|k|^2t}\|\hat{f}_0\|^2_{L^2_v}.
\eq
By direct calculation, we have
\bma
|(G_1(t,k)\hat{f}_0,\psi_0)|^2\ge &\ \frac{1}{2}B_2^2|k|^2|h|^2 e^{2{\rm Re}\lambda_1(|k|)t}\sin^2({\rm Im}\lambda_1(|k|)t) -C|k|^4e^{-2\rho|k|^2t}\|\hat{f}_0\|^2_{L^2_v},\\
|(G_1(t,k)\hat{f}_0,\psi')|^2\ge
&\ \frac{1}{2}|k|^2|h|^2e^{2{\rm Re}\lambda_1(|k|)t}\sin^2({\rm Im}\lambda_1(|k|)t)-C|k|^4e^{-2\rho|k|^2t}\|\hat{f}_0\|^2_{L^2_v},\\
|(G_1(t,k)\hat{f}_0,\psi_4)|^2\ge &\ \frac{1}{2}B_2^2|k|^2|h|^2 e^{2{\rm Re}\lambda_1(|k|)t}\sin^2({\rm Im}\lambda_1(|k|)t) -C|k|^4e^{-2\rho|k|^2t}\|\hat{f}_0\|^2_{L^2_v},\\
\|P_1(G_1(t,k)\hat{f}_0)\|^2_{L^2_v}\ge &\ \frac{1}{2}|k|^4|h|^2e^{2{\rm Re}\lambda_1(|k|)t}\sin^2({\rm Im}\lambda_1(|k|)t)\|L^{-1}P_1 (\hat{v}_1 \psi_1)\|_{L^2_v}-C|k|^6e^{-2\rho|k|^2t}\|\hat{f}_0\|^2_{L^2_v}.
\ema
Then, similarly, repeating the above steps, we can verify \eqref{Decay estimate lower limit 4}--\eqref{Decay estimate lower limit 5}. This proves the theorem.
\end{proof}

\section{The original nonlinear problem}
\label{sect4}
\setcounter{equation}{0}
In this section, we prove the long time decay rates of the solution to the Cauchy problem for relativistic Boltzmann equation with help of the asymptotic behaviors of linearized problem established in Section 3.

\begin{proof}[\textbf{Proof of Theorem \ref{nonlinear upper}}]
Let $f$ be a solution to RB equation \eqref{rb2} for $t>0$. We can represent this solution in terms of the semigroup $e^{tB}$ as
\bq\label{semigroup decompose}
f(t)=e^{tB}f_0+\int_0^te^{(t-s)B}\Gamma(f,f)ds.
\eq
For this global solution $f$, we define two functionals $Q_1(t)$ and $Q_2(t)$ for any $t>0$ as
\bmas
Q_1(t)=\sup_{0\le s\le t}\sum_{|\alpha|=0}^1
&\ \{(1+s)^{\frac{3}{4}+\frac{|\alpha|}{2}}\sum_{j=0}^4\|\dxa(f(s),\psi_j)\|_{L_x^2}+(1+s)^{\frac{5}{4}+\frac{|\alpha|}{2}}\|\dxa P_1f(s)\|_{\lxvtwo}\\
&\ +(1+s)^{\frac{5}{4}}(\|P_1f(s)\|_{ H_{N,1}}+\|\nabla_xP_0f(s)\|_{H^{N-1}})\},
\emas
and
\bmas
Q_2(t)=\sup_{0\le s\le t}\sum_{|\alpha|=0}^1
&\ \{(1+s)^{\frac{5}{4}+\frac{|\alpha|}{2}}\sum_{j=0}^4\|\dxa(f(s),\psi_j)\|_{L_x^2}+(1+s)^{\frac{7}{4}+\frac{|\alpha|}{2}}\|\dxa P_1f(s)\|_{\lxvtwo}\\
&\ +(1+s)^{\frac{7}{4}}(\|P_1f(s)\|_{ H_{N,1}}+\|\nabla_xP_0f(s)\|_{H^{N-1}})\}.
\emas

We claim that it holds under the assumptions of Theorem \ref{nonlinear upper} that
\bq\label{nonQ1}
Q_1(t)\le C\delta_0,
\eq
and if it is further satisfied $P_0f_0=0$ that
\bq\label{nonQ2}
Q_2(t)\le C\delta_0.
\eq
It is easy to verify that the estimates \eqref{nonlinear upper1}--\eqref{nonlinear upper3} and \eqref{nonlinear upper4}--\eqref{nonlinear upper6} from \eqref{nonQ1} and \eqref{nonQ2} respectively.

First of all, we prove the claim \eqref{nonQ1} as follows. From \cite{Yang1}, it holds that for any  $\alpha\in [0,1]$,
  \be
  \|\nu^{-\alpha}\Gamma(f,g)\|_{L^{2}_{v} }\leq C (\|\nu^{1-\alpha}f\|_{L^{2}_{v} }\|g\|_{L^{2}_{v} }+\|f\|_{L^{2}_{v} }\|\nu^{1-\alpha}g\|_{L^{2}_{v}}).
  \ee
Thus, we can estimate the nonlinear term $\Gamma(f,f)$ for $0\le s\le t$ in the terms of $Q_1(t)$ as
\bmas
\|\Gamma(f,f)\|_{\lxvtwo}&\ \le C\|\nu f\|_{Z_3}\|f\|_{Z_6}\le C(1+s)^{-2}Q_1(t)^2,\\
\|\Gamma(f,f)\|_{Z_1}&\ \le C\|\nu f\|_{\lxvtwo}\|f\|_{\lxvtwo}\le C(1+s)^{-\frac{3}{2}}Q_1(t)^2,
\emas
and
\bmas
\|\nabla_x\Gamma(f,f)\|_{\lxvtwo}&\ \le C\|\nu \nabla_xf\|_{Z_3}\|f\|_{Z_6}+\|\nu f\|_{Z_6}\|\nabla_xf\|_{Z_3}\\
&\ \le C(1+s)^{-\frac52}Q_1(t)^2,\\
\|\nabla_x\Gamma(f,f)\|_{Z_1}&\ \le C\|\nu \nabla_xf\|_{\lxvtwo}\|f\|_{\lxvtwo}+\|\nu f\|_{\lxvtwo}\|\nabla_xf\|_{\lxvtwo}\\
&\ \le C(1+s)^{-2}Q_1(t)^2.
\emas
Since the nonlinear term $\Gamma(f,f)$ satisfies $P_0\Gamma(f,f)=0$, we obtain the long decay rate of the macroscopic part by \eqref{Decay estimate upper limit 1} and \eqref{Decay estimate upper limit 3} as follow:
\bma\label{nonlinear upper f}
\|(f(t),\psi_j)\|_{\lxtwo} \le &\ C(1+t)^{-\frac{3}{4}}(\|f_0\|_{\lxvtwo}+\|f_0\|_{Z_1})\nnm\\
&\ +C\int_0^t(1+t-s)^{-\frac{5}{4}}(\|\Gamma(f,f)\|_{\lxvtwo}+\|\Gamma(f,f)\|_{Z_1})ds\nnm\\
\le &\ C\delta_0(1+t)^{-\frac{3}{4}}+C\int_0^t(1+t-s)^{-\frac{5}{4}}(1+s)^{-\frac{3}{2}}Q_1(t)^2ds\nnm\\
\le &\ C\delta_0(1+t)^{-\frac{3}{4}}+C(1+t)^{-\frac{5}{4}}Q_1(t)^2,
\ema
and
\bma\label{nonlinear upper df}
\|\nabla_x(f(t),\psi_j)\|_{\lxtwo} \le &\ C(1+t)^{-\frac{5}{4}}(\|\nabla_xf_0\|_{\lxvtwo}+\|f_0\|_{Z_1})\nnm\\
&\ +C\int_0^t(1+t-s)^{-\frac{5}{4}}(\|\nabla_x\Gamma(f,f)\|_{\lxvtwo}+\|\Gamma(f,f)\|_{Z_1})ds\nnm\\
\le &\ C\delta_0(1+t)^{-\frac{5}{4}}+C\int_0^t(1+t-s)^{-\frac{5}{4}}(1+s)^{-\frac{3}{2}}Q_1(t)^2ds\nnm\\
\le &\ C\delta_0(1+t)^{-\frac{5}{4}}+C(1+t)^{-\frac{5}{4}}Q_1(t)^2.
\ema

Then, we estimate the time-decay rate of microscopic part $P_1f(t)$ as follows.
We obtain from \eqref{Decay estimate upper limit 2} and \eqref{Decay estimate upper limit 4} that
\bma\label{nonlinear upper p1f}
\|P_1f(t)\|_{\lxvtwo} \le &\ C(1+t)^{-\frac{5}{4}}(\|f_0\|_{\lxvtwo}+\|f_0\|_{Z_1})\nnm\\
&\ +C\int_0^t(1+t-s)^{-\frac{7}{4}}(\|\Gamma(f,f)\|_{\lxvtwo}+\|\Gamma(f,f)\|_{Z_1})ds\nnm\\
\le &\ C\delta_0(1+t)^{-\frac{5}{4}}+C\int_0^t(1+t-s)^{-\frac{7}{4}}(1+s)^{-\frac{3}{2}}Q_1(t)^2ds\nnm\\
\le &\ C\delta_0(1+t)^{-\frac{5}{4}}+C(1+t)^{-\frac{3}{2}}Q_1(t)^2,
\ema
and
\bma\label{nonlinear upper dp1f}
\|\nabla_xP_1f(t)\|_{\lxvtwo} \le &\ C(1+t)^{-\frac{7}{4}}(\|\nabla_xf_0\|_{\lxvtwo}+\|f_0\|_{Z_1})\nnm\\
&\ +C\int_0^{\frac{t}{2}}(1+t-s)^{-\frac{9}{4}}(\|\nabla_x\Gamma(f,f)\|_{\lxvtwo}+\|\Gamma(f,f)\|_{Z_1})ds\nnm\\
&\ +C\int_{\frac{t}{2}}^t(1+t-s)^{-\frac{7}{4}}(\|\nabla_x\Gamma(f,f)\|_{\lxvtwo}+\|\nabla_x\Gamma(f,f)\|_{Z_1})ds\nnm\\
\le &\ C\delta_0(1+t)^{-\frac{7}{4}}+C\int_0^{\frac{t}{2}}(1+t-s)^{-\frac{9}{4}}(1+s)^{-\frac{3}{2}}Q_1(t)^2ds\nnm\\
&\ +C\int_{\frac{t}{2}}^t(1+t-s)^{-\frac{7}{4}}(1+s)^{-2}Q_1(t)^2ds\nnm\\
\le &\ C\delta_0(1+t)^{-\frac{7}{4}}+C(1+t)^{-2}Q_1(t)^2.
\ema

With the help of the priori estimates \eqref{nonlinear upper f}--\eqref{nonlinear upper dp1f}, we can verify the claim \eqref{nonQ1}. Indeed, by direct computation (cf.\cite{Yang1}), there are two functionals $H_1$ and $D_1$ related to the global solution $f$:
\be\bln\label{hl dl expression}
H_1(f)&\ \thicksim\sum_{|\alpha|\le N}\|\nu \da P_1f\|^2_{\lxvtwo}+\sum_{|\alpha|\le N-1}\|\da\nabla_xP_0f\|^2_{\lxvtwo},\\
D_1(f)&\ \thicksim\sum_{|\alpha|\le N}\|\nu^{\frac{3}{2}} \da P_1f\|^2_{\lxvtwo}+\sum_{|\alpha|\le N-1}\|\da\nabla_xP_0f\|^2_{\lxvtwo},
\eln\ee
such that
\bq
H_1(f)\le CD_1(f),
\eq
and
\bq\label{hl and dl}
\frac{d}{dt}H_1(f(t))+\mu D_1(f(t))\le C\|\nabla_xP_0f(t)\|^2_{\lxvtwo}.
\eq
This together with \eqref{nonlinear upper df} leads to
\bma\label{nonlinear upper hl}
H_1(f(t))&\ \le e^{-C\mu t}H_1(f_0)+\int_0^te^{-C\mu(t-s)}\|\nabla_xP_0f(s)\|^2_{\lxvtwo}ds\nnm\\
&\ \le C\delta_0^2 e^{-C\mu t}+\int_0^te^{-C\mu(t-s)}(1+s)^{-\frac{5}{2}}(\delta_0+Q_1(t)^2)^2ds\nnm\\
&\ \le C(1+t)^{-\frac{5}{2}}(\delta_0+Q_1(t)^2)^2.
\ema
Making summing to \eqref{nonlinear upper f}--\eqref{nonlinear upper dp1f} and \eqref{nonlinear upper hl}, we obtain
$$Q_1(t)\le C\delta_0+CQ_1(t)^2,$$
from which the claim \eqref{nonQ1} can be verified provided that $\delta_0>0$ is small enough.

Next, we turn to prove the claim \eqref{nonQ2} for the case $P_0f_0=0$ as follows. Indeed, if it holds $P_0f_0=0$, the macroscopic density, momentum, energy and their spatial derivatives can be estimated as follows:
\bma\label{nonlinear upper f2}
\|(f(t),\psi_j)\|_{\lxtwo}\le &\ C(1+t)^{-\frac{5}{4}}(\|f_0\|_{\lxvtwo}+\|f_0\|_{Z_1})\nnm\\
&\ +C\int_0^t(1+t-s)^{-\frac{5}{4}}(\|\Gamma(f,f)\|_{\lxvtwo}+\|\Gamma(f,f)\|_{Z_1})ds\nnm\\
\le &\ C\delta_0(1+t)^{-\frac{5}{4}}+C\int_0^t(1+t-s)^{-\frac{5}{4}}(1+s)^{-\frac{5}{2}}Q_2(t)^2ds\nnm\\
\le &\ C\delta_0(1+t)^{-\frac{5}{4}}+C(1+t)^{-\frac{5}{4}}Q_2(t)^2,
\ema
and
\bma\label{nonlinear upper df2}
\|\nabla_x(f(t),\psi_j)\|_{\lxtwo}\le &\ C(1+t)^{-\frac{7}{4}}(\|f_0\|_{\lxvtwo}+\|f_0\|_{Z_1})\nnm\\
&\ +C\int_0^t(1+t-s)^{-\frac{7}{4}}(\|\nabla_x\Gamma(f,f)\|_{\lxvtwo}+\|\Gamma(f,f)\|_{Z_1})ds\nnm\\
\le &\ C\delta_0(1+t)^{-\frac{7}{4}}+C\int_0^t(1+t-s)^{-\frac{7}{4}}(1+s)^{-\frac{5}{2}}Q_2(t)^2ds\nnm\\
\le &\ C\delta_0(1+t)^{-\frac{7}{4}}+C(1+t)^{-\frac{7}{4}}Q_2(t)^2,
\ema
where we have used
\bgrs
\|\Gamma(f,f)\|_{\lxvtwo}+\|\Gamma(f,f)\|_{Z_1}\le C(1+s)^{-\frac{5}{2}}Q_2(t)^2,\\
\|\nabla_x\Gamma(f,f)\|_{\lxvtwo}+\|\Gamma(f,f)\|_{Z_1}\le C(1+s)^{-\frac{5}{2}}Q_2(t)^2.
\egrs
In terms of \eqref{Decay estimate upper limit 4} the microscopic part and its spatial derivative can be estimated as
\bma\label{nonlinear upper p1f2}
\|P_1f(t)\|_{\lxvtwo} \le &\ C(1+t)^{-\frac{7}{4}}(\|f_0\|_{\lxvtwo}+\|f_0\|_{Z_1})\nnm\\
&\ +C\int_0^t(1+t-s)^{-\frac{7}{4}}(\|\Gamma(f,f)\|_{\lxvtwo}+\|\Gamma(f,f)\|_{Z_1})ds\nnm\\
\le &\ C\delta_0(1+t)^{-\frac{7}{4}}+C\int_0^t(1+t-s)^{-\frac{7}{4}}(1+s)^{-\frac{5}{2}}Q_2(t)^2ds\nnm\\
\le &\ C\delta_0(1+t)^{-\frac{7}{4}}+C(1+t)^{-\frac{7}{4}}Q_2(t)^2,
\ema
and
\bma\label{nonlinear upper dp1f2}
\|\nabla_xP_1f(t)\|_{\lxvtwo} \le &\ C(1+t)^{-\frac{9}{4}}(\|\nabla_xf_0\|_{\lxvtwo}+\|f_0\|_{Z_1})\nnm\\
&\ +C\int_0^t(1+t-s)^{-\frac{9}{4}}(\|\nabla_x\Gamma(f,f)\|_{\lxvtwo}+\|\Gamma(f,f)\|_{Z_1})ds\nnm\\
\le &\ C\delta_0(1+t)^{-\frac{9}{4}}+C\int_0^t(1+t-s)^{-\frac{9}{4}}(1+s)^{-\frac{5}{2}}Q_2(t)^2ds\nnm\\
\le &\ C\delta_0(1+t)^{-\frac{9}{4}}+C(1+t)^{-\frac{9}{4}}Q_2(t)^2.
\ema
Therefore, with the help \eqref{nonlinear upper df2}, we can obtain by \eqref{hl and dl} that
\bma
H_1(f(t))&\ \le e^{-C\mu t}H_1(f_0)+\int_0^te^{-C\mu(t-s)}\|\nabla_xP_0f(s)\|^2_{\lxvtwo}ds\nnm\\
&\ \le C\delta_0^2 e^{-C\mu t}+\int_0^te^{-C\mu(t-s)}(1+s)^{-\frac{7}{2}}(\delta_0+Q_2(t)^2)^2ds\nnm\\
&\ \le C(1+t)^{-\frac{7}{2}}(\delta_0+Q_2(t)^2)^2.
\ema
This together with \eqref{hl dl expression} yields
$$Q_2(t)\le C\delta_0+CQ_2(t)^2,$$
which implies the claims \eqref{nonQ2} provided that $P_0f_0=0$ and $\delta_0>0$ is small enough.
\end{proof}

\begin{proof}[\textbf{Proof of Theorem \ref{nonlinear lower}}]
By \eqref{semigroup decompose}, Theorem \ref{Decay estimate upper limit} and Theorem \ref{nonlinear upper}, we can establish the lower bounds of the time decay rates of the macroscopic part and the microscopic part of the global solution $f$ when $t>0$ large enough that
\bma
\|(f(t),\psi_j)\|_{\lxtwo}&\ \ge\|(e^{tB}f_0,\psi_j)\|_{\lxtwo}-\int_0^t\|(e^{(t-s)B}\Gamma(f,f),\psi_j)\|_{\lxtwo}ds\nnm\\
&\ \ge C_1\delta_0(1+t)^{-\frac{3}{4}}-C_2\delta_0^2(1+t)^{-\frac{5}{4}},
\ema
and
\bma
\|P_1f(t)\|_{\lxvtwo}&\ \ge\|P_1(e^{tB}f_0)\|_{\lxvtwo}-\int_0^t\|P_1(e^{(t-s)B}\Gamma(f,f))\|_{\lxvtwo}ds\nnm\\
&\ \ge C_1\delta_0(1+t)^{-\frac{5}{4}}-C_2\delta_0^2(1+t)^{-\frac{3}{2}},
\ema
from which and Theorem \ref{nonlinear upper}, we can obtain
\bma
\|f(t)\|_{ H_{N,1}}&\ \ge\|P_0f(t)\|_{\lxvtwo}-\|\nu P_1f(t)\|_{\lxvtwo}-\sum_{1\le |\alpha|\le N}\|\nu\dxa f(t)\|_{\lxvtwo}\nnm\\
&\ \ge 5C_1\delta_0(1+t)^{-\frac{3}{4}}-5C_2\delta_0^2(1+t)^{-\frac{5}{4}}-C_3\delta_0(1+t)^{-\frac{5}{4}}.
\ema
This gives rise to \eqref{nonlinear lower1}--\eqref{nonlinear lower3} for sufficiently $t>0$ and $\delta_0>0$ small enough. By repeating similar arguments, we can prove \eqref{nonlinear lower4}--\eqref{nonlinear lower6}, the detail are omitted .
\end{proof}

\end{document}